\documentclass[10pt,a4paper]{article}
\usepackage[a4paper, left=3cm, right=3cm, top=3cm, bottom=3cm]{geometry}
\usepackage{graphicx} 
\usepackage{amsmath}
\usepackage{mathtools}
\usepackage{todonotes}
\usepackage{tcolorbox}
\usepackage{float}
\tcbuselibrary{breakable}
\usepackage{amsfonts}
\usepackage{amsthm}
\usepackage{subcaption}
\usepackage{multirow}
\usepackage{lipsum}  
\usepackage{algorithm}
\usepackage{algpseudocode}
\usepackage{mathabx}
\usepackage{bm}
\usepackage{cite}
\usepackage{hyperref}
\usepackage{authblk}
\usepackage{enumitem}

\theoremstyle{theorem}
\newtheorem{theorem}{Theorem}[section]
\newtheorem{remark}{Remark}[section]
\newtheorem{lemma}{Lemma}[section]

\newtheorem{example}{Example}[section]

\newcommand{\lp}{\reflectbox{$'$}}
\numberwithin{equation}{section}

\DeclareMathOperator*{\argmin}{argmin}

\graphicspath{{../Figures/}}

\begin{document}

\title{Collective Annealing by Switching Temperatures:\\ a Boltzmann-type framework}
\author[1,2]{Frédéric Blondeel\thanks{frederic.blondeel@unife.it}}
\author[2,3]{Lorenzo Pareschi\thanks{l.pareschi@hw.ac.uk}}
\author[1]{Giovanni Samaey\thanks{giovanni.samaey@kuleuven.be}}

\affil[1]{Department of Computer Science, KU Leuven, Belgium}
\affil[2]{Department of Mathematics and Computer Science University of Ferrara, Italy}
\affil[3]{Heriot-Watt University, Edinburgh, UK}
\date{}

\maketitle

\begin{abstract}
The design of effective cooling strategies is a crucial component in simulated annealing algorithms based on the Metropolis method. Traditionally, this is achieved through inverse logarithmic decays of the temperature to ensure convergence to global minima. In this work, we propose Collective Annealing by Switching Temperatures ({CAST}), a novel collective simulated annealing dynamic in which agents interact to learn an adaptive cooling schedule. Inspired by the particle-swapping mechanism of parallel tempering, we introduce a Boltzmann-type framework in which particles exchange temperatures through stochastic binary interactions. Under suitable conditions on the interaction parameters, this process induces a 
monotone decrease of the expected average temperature in the system.
Numerical results indicate that the proposed approach can improve 
convergence speed over classical simulated annealing on multimodal 
benchmark problems, especially in regimes where adaptive exploration is important.
\end{abstract}


\tableofcontents

\section{Introduction}
Global optimization is inherently difficult, primarily due to two factors.
First, there is no guarantee that a computed solution corresponds to the global minimum rather than a local one, particularly in nonconvex landscapes characterized by multiple attractors.
Second, in many practical settings, gradient information may be unavailable, unreliable, or too costly to evaluate, and the objective function itself may not admit an explicit analytic representation.
These issues make conventional gradient-based schemes, such as stochastic gradient descent and its variants, inapplicable or prone to premature convergence.
In such settings, a popular alternative involves the use of \textit{metaheuristic algorithms}: stochastic, population-based methods designed to explore complex and multimodal energy landscapes. Classical examples include Simulated Annealing (SA)~\cite{kirkpatrick_optimization_1983}, Genetic Algorithms (GA)~\cite{holland_adaptation_1992,goldberg_genetic_1989}, Ant Colony Optimization (ACO)~\cite{dorigo_ant_1996}, Particle Swarm Optimization (PSO)~\cite{kennedy_particle_1995}, and more recently, Consensus-Based Optimization (CBO)~\cite{carrillo_analytical_2018, pinnau_consensus-based_2017}. 
These algorithms do not rely on gradient information and employ sophisticated exploration-exploitation heuristics to guide the system towards a probable global minimum. 

Simulated Annealing was introduced in the early 1980s by Kirkpatrick, Gelatt, and Vecchi~\cite{kirkpatrick_optimization_1983}. Originally designed to solve combinatorial optimization problems like the Traveling Salesman Problem, SA has since found applications across engineering, machine learning, and physics\cite{gendreau_simulated_2019,rere_simulated_2015}. The method is inspired by the metallurgical annealing process, where a material is slowly cooled to reduce its Gibbs free energy and achieve a low-energy crystalline structure. Computationally, SA implements this idea via a Metropolis--Hastings Markov Chain Monte Carlo (MCMC) dynamic^^>\cite{hastings_monte_1970,metropolis_equation_1953} combined with a global cooling schedule. At high temperatures, the system accepts both better and worse solutions, enabling wide exploration. As the temperature decreases, the system becomes more selective, favoring low-energy configurations. Under ideal conditions, convergence to the global minimum is guaranteed by an inverse-logarithmic cooling schedule~\cite{hajek_cooling_1988}, although such schedules are often impractical in applications due to slow convergence rates.

Parallel Tempering (PT), also known as Replica Exchange Monte Carlo, was originally developed to enhance MCMC sampling\cite{lelievre_computation_2007, stoltz_longtime_2018, voter_parallel_1998}. PT generalizes SA by introducing multiple replicas of the system, each evolving at a distinct temperature~\cite{geyer_markov_1991, earl_parallel_2005,marinari_simulated_1992}. High-temperature replicas explore the energy landscape more freely, while low-temperature replicas refine the search near local minima. By allowing periodic swaps of configurations between replicas, PT improves mixing and avoids trapping in poor minima.

In traditional parallel tempering, interactions among replicas are performed through Metropolis-type swaps of states or temperatures. 
In this work, we propose Collective Annealing by Switching Temperatures (CAST), a generalization in which temperature exchange is modeled as a stochastic binary interaction between particles. 
In contrast to other collective dynamics, the binary exchange model offers a computationally efficient scheme, with cost scaling linearly in the number of samples due to its localized pairwise interactions. 
Our approach is conceptually related to thermostatting techniques such as the Nosé–Hoover method^^>\cite{nose_unified_1984,hoover_canonical_1985}, which enforce canonical sampling by coupling the system to an auxiliary dynamical variable. 
Inspired by kinetic theory and Boltzmann-type models, our approach introduces gradual, probabilistic temperature exchanges, enabling a collective and adaptive cooling strategy.

We also mention recent extensions of consensus-based optimization to constrained,
bilevel, and multilevel problems, where additional particle populations or
auxiliary variables are introduced to encode the structure of the optimization
task \cite{trillos_cb2o_2025,herty_multiscale_2025}. The present approach
is different in scope: CAST remains a single-level global optimization method, but
it augments the particle system with a dynamically evolving temperature variable.
In this sense, it is close in spirit to multiscale particle methods in which the
optimization dynamics is coupled to an auxiliary adaptive mechanism.

A common drawback of many metaheuristic algorithms is the lack of a rigorous theoretical foundation, as they are typically formulated only at the particle level. To investigate the theoretical properties of our method, we employ tools from classical kinetic theory to characterize the behavior of the system in the limit of a large number of particles^^>\cite{pareschi_optimization_2024, borghi_kinetic_2024, pareschi_interacting_2014,albi_vehicular_2019}. We refer also to the recent work^^>\cite{herty_kinetic_2026} on kinetic formulations of simulated annealing with entropy-based cooling strategies. 
We derive a kinetic formulation that provides a useful mesoscopic description of 
the collective dynamics and allows us to establish structural properties such as 
mass preservation, positivity under suitable conditions, and moment identities for 
the temperature distribution. Specifically, we describe the evolution of the joint probability density $f(x,T,t)$ of particles with position $x$ and temperature $T$ at time $t$ through an equation of the form
\begin{equation}
\label{eq:SAPTstrong}
\frac{\partial f}{\partial t} = \mathcal{L}(f) + Q(f,f),
\end{equation}
where $\mathcal{L}(f)$ is a linear operator governing spatial exploration based on a Metropolis search algorithm, and $Q(f,f)$ is a Boltzmann-type  operator modeling stochastic temperature interactions. Let us mention that Boltzmann-type operators in global optimization have been proposed also in \cite{benfenati_binary_2022,albi_kinetic_2023,borghi_kinetic_2025}.
The detailed expressions for $\mathcal{L}(f)$ and $Q(f,f)$ will be developed in the following sections. 

In particular, Section \ref{sec:cast} introduces the novel method, by first considering the algorithm at the particle-level. We introduce the Metropolis search operator and particle interaction dynamics. Next, we switch to the kinetic interpretation of the algorithm. 
A formal analysis is then carried out in Section \ref{sec:analysis}, where we derive evolution 
equations for suitable marginals of the probability density function, establish 
temperature moment identities, and discuss formal asymptotic mean-field  
limits. Finally, in Section \ref{sec:impl} and \ref{sec:num}, we give some implementation details and choices in order to present a series of numerical experiments that validate the proposed method and compare its performance with classical Simulated Annealing on a set of benchmark problems. We conclude with a section outlining key findings and directions for future research.

\section{Collective Annealing by Switching Temperatures}
\label{sec:cast}
In this section, we describe the proposed algorithm. We start with the particle level description and then give the corresponding kinetic formulation of the collective behaviour.
\subsection{The stochastic particle based algorithm}
\label{sec:particle}
Let us consider the general problem of computing the global minimum
\begin{equation}
x^*=\argmin_{x\in\mathcal{D}}\mathcal{F}(x),
\end{equation}
where $\mathcal{D}\subseteq \mathbb{R}^d$, $d\geq 1$, and $\mathcal{F}(x):\mathcal{D}\rightarrow\mathbb{R}$ is a continuous function
which admits one
global minimizer. We denote
this minimizer as $x^*$.

We start by presenting the proposed algorithm on the particle level. We initialize CAST as follows: 
\begin{enumerate}
\item A set of $N$ particles $X_1^0,\ldots,X_N^0$ is randomly placed on the search space.
\item Each particle has its own {nonnegative} temperature $T_1^0,\ldots,T_N^0$.
\end{enumerate}
Then, in each time step $n$, CAST performs two actions:
\begin{itemize}
\item \textbf{Exploration}. This is essentially fixed-temperature multi-particle Simulated Annealing and will be discussed in section \ref{subsec:particle_Lf}. This differs from classical SA, where a single particle is cooled by a fixed schedule.
\item \textbf{Particle interaction}. The particles are allowed to exchange temperatures, allowing for a temperature decay. This dynamic will be discussed in section \ref{subsec:particle_Qff}. This differs from classical PT, where particles fully swap temperatures.
\end{itemize}

\subsubsection{Exploring the search space}
\label{subsec:particle_Lf}
The exploration phase of the algorithm goes as follows:
\begin{enumerate}
\item Generate a candidate move as
\begin{equation}
\label{eq:SAcandidate}
\tilde{X_i}^{n+1}=X_i^n+\eta(T_i^n)\xi_i^n,\quad i=1,\ldots,N,
\end{equation}
where $\xi_i^n\sim p(\xi, T^n_i)$, $\eta(T_i^n)$ is a function of the temperature and $p(\xi, T^n_i)$ is some probability density (possibly) depending on the temperature. We assume the probability density to be symmetric and centered in zero.
\item Check whether the candidate is in a better position than before. If so, the move is accepted. If not, it can still be accepted in accordance to a temperature dependent Boltzmann--Gibbs probability. The move is rejected otherwise.
\begin{equation}
\label{eq:SAacceptreject}
\begin{dcases}
\text{Accepted}& \text{if }\mathcal{F}(\tilde{X}_i^{n+1})<\mathcal{F}(X_i^n),\\
\text{Accepted with prob. } \exp\Bigg(-\frac{\mathcal{F}(\tilde{X}_i^{n+1})-\mathcal{F}(X_i^n)}{T_i^n}\Bigg)& \text{if }\mathcal{F}(\tilde{X}_i^{n+1})\geq\mathcal{F}(X_i^n),\\
\text{Rejected otherwise.}
\end{dcases}
\end{equation}
\end{enumerate}

\begin{remark}[Step noise]
\label{rem:noise}
Many options exist in the choice of $\eta(T)$ and $p(\xi, T)$. A classical choice is $p(\xi,T)\sim\mathcal{N}(0,1)$ and $\eta(T)=\sqrt{2T}$ so that the temperature is related to the variance of the random variable $\eta(T)\xi$. The finite mean and variance allows for a passage to Langevin dynamics characterized by a Fokker-Planck equation^^>\cite{pareschi_optimization_2024, chizat_mean-field_2022}.
Another option, as shown in^^>\cite{szu_fast_1987, xavier-de-souza_coupled_2010}, is Cauchy noise, i.e., \[p(\xi, T)=\frac{1}{\pi}\frac{T}{T^2+\xi^2}\] with $\eta(T)=1$. Although Cauchy noise has an undefined mean and variance, meaning we cannot perform a mean-field scaling as shown later, the fat Lorentzian tails enable larger jumps which permit to outperform the classical SA algorithm. It is important to underline that our kinetic formulation remains valid also under this choice of noise. For this reason, we will mostly focus on Cauchy noise in our numerical examples against test functions in Section \ref{sec:num}. We mention here that the positive effects of larger jumps has been observed also in other recently developed collective particle search algorithms^^>\cite{kalise_consensus-based_2023,borghi_swarm-based_2025}.
\end{remark}

\begin{remark}[Classical Simulated Annealing]
What we have set forth until now is essentially multi-particle SA without cooling. Our temperature dynamic will be introduced in the sequel. But to give some context, in classical multi-particle SA all particles have the same temperature. This means that the temperature of each particle is reduced at the end of each move in the search space, whether it was accepted or not. The temperature is reduced according to some predefined rule of the form
\begin{equation}
T^n_i=\theta_n T^0_i, \quad\theta\in(0,1),\quad \forall i,
\end{equation}
where the initial temperatures $T^0_i$ are strictly positive. A classical choice is
\begin{equation}
\label{eq:temp}
\theta_n=\frac{1}{\ln(n+e)},
\end{equation}
where the algorithm is proved to converge^^>\cite{hajek_cooling_1988,pareschi_optimization_2024}. However, in more practical applications, a geometric reduction where 
\begin{equation}
\theta_n=(\alpha)^n
\end{equation} 
with $\alpha \in (0,1)$ has shown good potential.
\end{remark}

\subsubsection{Non-local binary interactions and temperature exchanges}
\label{subsec:particle_Qff}
Our particle interaction dynamic is a generalization of classical parallel tempering^^>\cite{earl_parallel_2005}. This is achieved by introducing an alignment process and adding stochasticity. 

Grouping the particles pairwise, and for all $N/2$ distinct pairs of particles $(i,j)$ without replacement, we set $x\coloneqq X^n_i$,$x_*\coloneqq X^n_j$, $T\coloneqq T^n_i$ and $T_*\coloneqq T^n_j$
Then, we apply the interaction described below and set $T^{n+1}_i\eqqcolon T'$ and $T^{n+1}_j\eqqcolon T'_*$.

The non-local temperature interaction between a pair of particles $(x,T)$ and $(x_*,T_*)$ yielding $(x,T')$ and $(x_*,T'_*)$ is described as 
\begin{equation}
\label{eq:PT_trans}
\begin{split}
T' &= T-\lambda(T-T_*)\chi_\mathcal{F}(x,x_*,T,T_*)-\mu(T-T_*)\chi_\mathcal{F}(x_*,x,T_*,T)+\sigma\zeta\\
T_*' &= T_*-\lambda(T_*-T)\chi_\mathcal{F}(x_*,x,T_*,T)-\mu(T_*-T)\chi_\mathcal{F}(x,x_*,T,T_*)+\sigma_*\zeta_*,
\end{split}
\end{equation}
where $\lambda,\mu\in [0,1]$, $\sigma = \sigma(x,x_*,T,T_*)$, $\sigma_* = \sigma(x_*,x,T_*,T)$ define the intensity of the noise characterized by the random variables $\zeta$, $\zeta_*$ assumed symmetric with mean zero and where
\begin{equation*}
\chi_\mathcal{F}(x,x_*,T,T_*)=\Psi\big(\mathcal{F}(x)<\mathcal{F}(x_*)\big)\Psi(T_*<T),
\end{equation*}
is the interaction-indicator function with $\Psi(\cdot)$ the classical indicator function. 
Essentially, the interaction-indicator will drive the dynamic to promote exploration for poorly performing particles and exploitation for well-performing ones. We illustrate with an example.
\begin{example}
\label{ex:viewpoint}
We consider two particles, each with their respective temperatures, i.e., $(x,T)$ and $(x_*,T_*)$. Due to symmetry, it suffices to only look at terms with $\chi_\mathcal{F}(x,x_*,T,T_*)$. For now, we ignore the effect of noise. For interaction to occur, the interaction-indicator gives two conditions we consider to be satisfied, namely $\mathcal{F}(x)<\mathcal{F}(x_*)$ and $T_*<T$. We get
\begin{align*}
T' &=T-\lambda(T-T_*) \ =T-\mu T+\lambda T_* - (\lambda-\mu) T\\
T_*'&=\underbrace{T_*+\mu(T-T_*)}_{VP1}=\underbrace{T_*+\mu T -\lambda T_* +(\lambda - \mu)T_*}_{VP2}.
\end{align*}
We discuss two viewpoints:
\begin{enumerate}[leftmargin=*, align=left]
\item[$VP1$.] \textbf{The warming/cooling viewpoint.} \ The well-performing particle cools down by a factor $\lambda$ of the difference of temperatures. The poorly performing particle warms up by a factor $\mu$ of the difference of temperatures.
\item [$VP2.$]  \textbf{The exchange viewpoint.} \ The particles exchange parts of their temperatures to each other and gain or lose some from the ``environment". Figure \ref{fig:vis} illustrates this viewpoint.

\begin{figure}[htb]

\centering

\scalebox{0.9}{
\tikzset{every picture/.style={line width=0.75pt}} 

\begin{tikzpicture}[x=0.75pt,y=0.75pt,yscale=-1,xscale=1]

\draw  (103,254.28) -- (533,254.28)(146,25) -- (146,279.75) (526,249.28) -- (533,254.28) -- (526,259.27) (141,32) -- (146,25) -- (151,32)  ;
\draw    (146,131.33) .. controls (158.44,102) and (181.78,35.33) .. (226.44,140) .. controls (271.11,244.67) and (302.17,255.08) .. (338.17,151.08) .. controls (374.17,47.08) and (396.44,189.33) .. (427.11,142) .. controls (457.78,94.67) and (481.11,102) .. (493.11,102.67) ;
\draw  [color={rgb, 255:red, 0; green, 0; blue, 0 }  ,draw opacity=1 ][fill={rgb, 255:red, 0; green, 0; blue, 0 }  ,fill opacity=1 ] (287.67,223.92) .. controls (287.67,222.95) and (288.45,222.17) .. (289.42,222.17) .. controls (290.38,222.17) and (291.17,222.95) .. (291.17,223.92) .. controls (291.17,224.88) and (290.38,225.67) .. (289.42,225.67) .. controls (288.45,225.67) and (287.67,224.88) .. (287.67,223.92) -- cycle ;
\draw  [color={rgb, 255:red, 0; green, 0; blue, 0 }  ,draw opacity=1 ][fill={rgb, 255:red, 0; green, 0; blue, 0 }  ,fill opacity=1 ] (413.67,151.92) .. controls (413.67,150.95) and (414.45,150.17) .. (415.42,150.17) .. controls (416.38,150.17) and (417.17,150.95) .. (417.17,151.92) .. controls (417.17,152.88) and (416.38,153.67) .. (415.42,153.67) .. controls (414.45,153.67) and (413.67,152.88) .. (413.67,151.92) -- cycle ;
\draw [color={rgb, 255:red, 126; green, 211; blue, 33 }  ,draw opacity=1 ]   (288.17,194.58) .. controls (289.65,63.9) and (378.86,35.15) .. (410.72,127.74) ;
\draw [shift={(411.67,130.58)}, rotate = 252.1] [fill={rgb, 255:red, 126; green, 211; blue, 33 }  ,fill opacity=1 ][line width=0.08]  [draw opacity=0] (8.93,-4.29) -- (0,0) -- (8.93,4.29) -- cycle    ;
\draw [color={rgb, 255:red, 126; green, 211; blue, 33 }  ,draw opacity=1 ]   (413.17,185.08) .. controls (410.2,198.94) and (392.04,256.91) .. (315.99,228.96) ;
\draw [shift={(313.67,228.08)}, rotate = 21.04] [fill={rgb, 255:red, 126; green, 211; blue, 33 }  ,fill opacity=1 ][line width=0.08]  [draw opacity=0] (8.93,-4.29) -- (0,0) -- (8.93,4.29) -- cycle    ;
\draw [color={rgb, 255:red, 208; green, 2; blue, 27 }  ,draw opacity=1 ]   (288.17,194.58) -- (247.22,126.65) ;
\draw [shift={(245.67,124.08)}, rotate = 58.92] [fill={rgb, 255:red, 208; green, 2; blue, 27 }  ,fill opacity=1 ][line width=0.08]  [draw opacity=0] (8.93,-4.29) -- (0,0) -- (8.93,4.29) -- cycle    ;
\draw [color={rgb, 255:red, 208; green, 2; blue, 27 }  ,draw opacity=1 ]   (413.17,185.08) -- (449.78,230.25) ;
\draw [shift={(411.17,183.08)}, rotate = 50.97] [fill={rgb, 255:red, 208; green, 2; blue, 27 }  ,fill opacity=1 ][line width=0.08]  [draw opacity=0] (8.93,-4.29) -- (0,0) -- (8.93,4.29) -- cycle    ;

\draw (270.33,226.73) node [anchor=north west][inner sep=0.75pt]    {$( x,T)$};
\draw (389.83,159.73) node [anchor=north west][inner sep=0.75pt]    {$( x_{*} ,T_{*})$};
\draw (105,22.4) node [anchor=north west][inner sep=0.75pt]    {$\mathcal{F}( x)$};
\draw (502.5,258.9) node [anchor=north west][inner sep=0.75pt]    {$x$};
\draw (335.67,53.9) node [anchor=north west][inner sep=0.75pt]    {$\mu T$};
\draw (378.17,231.9) node [anchor=north west][inner sep=0.75pt]    {$\lambda T_{*}$};
\draw (227.67,103.4) node [anchor=north west][inner sep=0.75pt]    {$( \lambda -\mu ) T$};
\draw (434.17,190.9) node [anchor=north west][inner sep=0.75pt]    {$( \lambda -\mu ) T_{*}{}_{\ }$};

\end{tikzpicture}
}

\caption{Visualization of the second viewpoint. Interaction between two particles $(x,T)$ and $(x_*,T_*)$, where the former sits in a better position. For interaction to occur $T>T_*$ must be true.}
\label{fig:vis}
\end{figure}
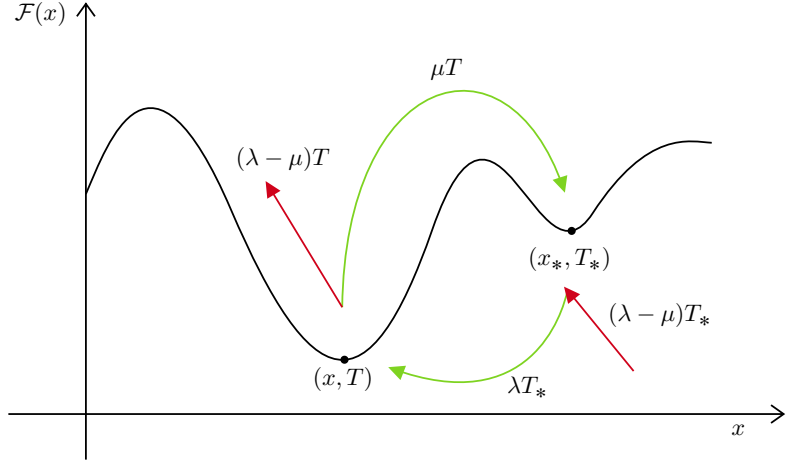
\end{enumerate}
A nice consequence of either viewpoint is the observation that if $\lambda=\mu$, then mean temperatures are conserved. In the warming/cooling viewpoint, this is a consequence of warming and cooling by precisely the same amount. In the exchange viewpoint, this is a consequence of having no gain or loss from the ``environment".

\end{example}
The following lemma can be easily derived.
\begin{lemma}
Let $\lambda, \mu \in [0,1]$ and let $\zeta$ and $\zeta_*$ be bounded random variables with values in $[-a,a]$ for some $a > 0$. Assume that for all $(x,x_*,T,T_*) \in \mathbb{R}^d \times \mathbb{R}^d \times \mathbb{R}^+ \times \mathbb{R}^+$, the interaction kernel $\sigma(x,x_*,T,T_*)$ satisfies the bound
\begin{equation}
\sigma(x,x_*,T,T_*) \leq  \frac{1}{a} \min\left\{T,T(1-\lambda)+\lambda T_*, \; T(1-\mu)+\mu T_*\right\},
\label{eq:pos}
\end{equation}
and that the same condition holds for
\(\sigma(x_*,x,T_*,T)\) after exchanging \(T\) and \(T_*\).
Then, the post-interaction temperatures \eqref{eq:PT_trans}  remain non-negative for all admissible values of $\zeta$, $\zeta_*$, $T$, and $T_*$.
\end{lemma}
Moreover, as we will show in Section \ref{subsec:analysis_moments}, the condition $\mu < \lambda$ also ensures that the expected mean temperature decreases over time.

\subsection{A kinetic description}
\label{sec:population}

In this subsection, we move to the description of the algorithm for large numbers of samples. This can be achieved by leveraging tools from classical kinetic theory^^>\cite{pareschi_interacting_2014} and considering the particle density $f=f(x,T,t)$ of samples in position $x$ with temperature $T$ at time $t>0$. We start from the evolution equation \eqref{eq:SAPTstrong}. First, we derive the form of the linear operator $\mathcal{L}(f)$. Second, we derive the form of the Boltzmann-type interaction operator $Q(f,f)$. Finally, we combine the two operators into a complete evolution equation. 

\subsubsection{Exploring the search space}
\label{subsec:population_Lf}

We define the action of the linear operator $\mathcal{L}(f)$ by ignoring any kind of interaction in $T$, i.e.,
\begin{equation}
\label{eq:Lf}
\frac{\partial f}{\partial t}=\mathcal{L}(f).
\end{equation}
We begin by generalizing the trial position given by eq. (\ref{eq:SAcandidate}) which now corresponds to 
\begin{equation}
\label{eq:SAexploration}
x'=x+\eta(T)\xi, \quad \xi\sim p(\xi, T),
\end{equation}
where $T\in\mathbb{R}^+$ is the temperature, $x\in\mathbb{R}^d$, and the proposal probability density $p(\xi, T)$ is assumed symmetric in $\xi$.

In the theoretical discussion below, the corresponding proposal kernel is understood to be symmetric on \(D\). Thus, if a trial move leaves \(D\), it is rejected, or equivalently the proposal is restricted to \(D\) in a symmetry-preserving way. We also assume that the Gibbs profile introduced below is normalizable for every \(T>0\).
 
Next, the acceptance-rejection dynamic of eq. (\ref{eq:SAacceptreject}) can be summarized in kernel form. 
For each fixed temperature \(T>0\), we introduce the Gibbs profile
\begin{equation}
\pi_{T,\mathcal F}(x)
=
\frac{1}{Z_{\mathcal F}(T)}
\exp\left(-\frac{\mathcal F(x)}{T}\right),
\qquad
Z_{\mathcal F}(T)
=
\int_{\mathcal D}
\exp\left(-\frac{\mathcal F(y)}{T}\right)\,dy.
\label{eq:bg_prob}
\end{equation}
The kernel now reads 
\begin{equation}
\label{eq:SAkernel}
\mathcal{B}_\mathcal{F}(x\rightarrow x')=\min\Bigg\{1,\frac{\pi_{T,\mathcal F}(x')}{\pi_{T,\mathcal F}(x)}\Bigg\}=
\min\left\{
1,
\exp\left(-\frac{\mathcal F(x')-\mathcal F(x)}{T}\right)
\right\}.
\end{equation}
If we define the temperature marginal 
\begin{equation}
F(T,t):=\int_{\mathcal D}f(x,T,t)\,dx,
\label{eq:tmarginal}
\end{equation}
then the exploration operator \(\mathcal L\)
does not change this marginal. Hence, for a prescribed temperature marginal \(F\), the
corresponding local equilibrium manifold of \(\mathcal L\) is
\begin{equation}
M_{\mathcal F}(x,T,t)=F(T,t)\pi_{T,\mathcal F}(x).
\label{eq:equilibrium}
\end{equation}
Finally, to write down the evolution equation of the form (\ref{eq:Lf}), we split the operator into a gain part $\mathcal{L}^+(f)$ and a loss part $\mathcal{L}^-(f)$, i.e.,
\begin{equation}
\label{eq:Lfsplit}
\frac{\partial f}{\partial t}=\mathcal{L}^+(f)-\mathcal{L}(f)^-.
\end{equation}
The gain and loss parts model the in- and out-flow of particles respectively. To ``count" the gain and loss particles, define 
\[
\Omega:=\mathcal{D}\times \mathbb{R}^+,\qquad \omega := (x,T)\in \Omega,\qquad d\omega=dx dT,
\]
and consider a volume element $d\omega$ around $\omega$. The number of particles undergoing a move, thus leaving the volume, is given by 
\begin{equation}
\label{eq:Lf-}
\mathcal{L}^-(f)d\omega=\mathbb{E}\left[\mathcal{B}_\mathcal{F}(x\rightarrow x')f(x,T,t)d\omega\right],
\end{equation}
where we have used the notation
\begin{equation}
\mathbb{E}\left[ g \right]=\int_{\mathbb{R}^d}g(\xi)p(\xi,T)d\xi
\end{equation}
to denote the expectation under the ``movement" distribution $p(\xi,T)$. Inversely, the number of particles entering the domain is given by
\begin{equation}
\label{eq:Lf+}
\mathcal{L}^+(f)d\omega=\mathbb{E}\left[\mathcal{B}_\mathcal{F}(x'\rightarrow x)f(x',T,t)dx'dT\right]
=\mathbb{E}\left[\mathcal{B}_\mathcal{F}(x'\rightarrow x)f(x',T,t)d\omega\right].
\end{equation}
Under these additional assumptions, one obtains the formal strong form
\begin{equation}
\label{eq:SAStrong}
\frac{\partial f(x,T,t)}{\partial t}=\mathbb{E}\left[\mathcal{B}_\mathcal{F}(x'\rightarrow x)f(x',T,t)-\mathcal{B}_\mathcal{F}(x\rightarrow x')f(x,T,t)\right]\coloneqq\mathcal{L}(f).
\end{equation}

\begin{remark}[Markov process viewpoint]
Since the exploration part of the algorithm is closely related to a Metropolis--Hastings-type algorithm, it is insightful to consider the Markov process viewpoint. Consider our particles to be states in such a process. Using classical notation, we wish to find the transition probability $P(x'|x)$ since it uniquely defines the process, i.e., 
\begin{equation*}
P(x\rightarrow x')=P(x'|x)P(x),
\end{equation*}
where $P(x\rightarrow x')$ is the distribution of states that have undergone transition and $P(x)$ is the distribution of states before transition. We know that
\begin{equation*}
P(x'|x)=A(x',x)g(x'|x)
\end{equation*} 
where $A(x',x)$ is the acceptance distribution and $g(x'|x)$ is the proposal distribution. The acceptance distribution gives the probability to accept the proposed state $x'$. The proposal distribution is the conditional probability of generating state $x'$ from state $x$. In our setting, the acceptance distribution is given by the kernel $\mathcal{B}_\mathcal{F}(x\rightarrow x')$. It gives the probability to accept the proposed state. The proposal distribution is given by $p(\xi,T)$ since it generates the state $x'$ from $x$. The distribution of states $x$ is given by $f(x,T,t)$. We now write
\begin{equation*}
f(x\rightarrow x',T,t,\xi)=\mathcal{B}_\mathcal{F}(x\rightarrow x')f(x,T,t)p(\xi,T).
\end{equation*}
To get rid of the dependence of $\xi$, we integrate over the random variable $\xi$. Finally, we get
\begin{equation*}
f(x\rightarrow x',T,t)=\int_{\mathbb{R}^d}\mathcal{B}_\mathcal{F}(x\rightarrow x')f(x,T,t)p(\xi,T)d\xi.
\end{equation*}
By simply multiplying both sides by $d\omega$, we recover exactly eq. (\ref{eq:Lf-}). By a similar computation, we can also recover eq. (\ref{eq:Lf+}).
\end{remark}

\begin{lemma}
\label{lemma:symm}
For any symmetric probability density $p(\xi,T)$ and any integrable function $g(x,x')$ we have
\begin{equation}
\mathbb{E}\left[\int_{\mathbb{R}^d}g(x',x)dx\right]=\mathbb{E}\left[\int_{\mathbb{R}^d}g(x,x')dx\right].
\end{equation}
\begin{proof}
\cite[Lemma 3.1]{pareschi_optimization_2024}
\end{proof}
\end{lemma}
Using this result, the evolution equation in eq. (\ref{eq:SAStrong}) is more conveniently written in weak form,
\begin{align}
\nonumber
\frac{\partial}{\partial t}\int_{\Omega}& f(x,T,t)\varphi(x,T)d\omega =\int_{\Omega} \mathcal{L}(f)\varphi(x,T)d\omega\\
\label{eq:SAweak}
&=\mathbb{E}\left[\int_{\Omega}\mathcal{B}_\mathcal{F}(x\rightarrow x',T)\big(\varphi(x',T)-\varphi(x,T)\big)f(x,T,t)d\omega \right]
\\
\nonumber
&=\frac{1}{2}\mathbb{E}\left[\int_{\Omega}\big(\varphi(x',T)-\varphi(x,T)\big)\big(\mathcal{B}_\mathcal{F}(x\rightarrow x',T)f(x,T,t) -\mathcal{B}_\mathcal{F}(x'\rightarrow x,T)f(x',T,t) \big)d\omega \right],
\end{align}
where we used Lemma \ref{lemma:symm} in the final step.

\subsubsection{Non-local binary interactions and temperature exchanges}
\label{subsec:population_Qff}

We define the action of the Boltzmann-type interaction operator $Q(f,f)$ by ignoring any kind of search space dynamic, i.e.,
\begin{equation}
\frac{\partial f}{\partial t}=Q(f,f).
\end{equation}
Since we are dealing with a binary interaction, the evolution equation will be nonlinear and characterized by a Boltzmann-like structure. Specifically, the resulting operator has the structure of a Boltzmann-Povzner interaction integral^^>\cite{povzner_boltzmann_1962}. Similar to what we did in the prequel, we can also derive the evolution equation by splitting the operator in a gain and loss part,
\begin{equation}
\label{eq:Qffsplit}
\frac{\partial f(x,T,t)}{\partial t} = Q^+(f,f)-Q^-(f,f).
\end{equation}
We consider the particles in a volume element $d\omega$ around $\omega$ leaving and entering. By assuming that particles interact with the same probability, the interaction kernel is unitary and the Boltzmann operator is reminiscent of the classical one for Maxwellian particles^^>\cite{pareschi_interacting_2014}. In the sequel we assume that $f$ is a probability density on $\Omega$. The number of particles undergoing an interaction of the form \eqref{eq:PT_trans}, thus leaving the volume is given by
\begin{equation}
\label{eq:Qff-}
\begin{split}
Q^-(f,f)d\omega&=\mathbb{E}\left[ \int_{\Omega} f(x,T,t)f(x_*,T_*,t)d\omega_* \right] d\omega\\
&= f(x,T,t)\underbrace{\int_{\Omega} f(x_*,T_*,t)d\omega_*}_{=1}  d\omega\\
&= f(x,T,t)d\omega.
\end{split}
\end{equation}
When the temperature interaction map is invertible and sufficiently regular, as a consequence of an interaction between $(x,\lp T)$ and $(x_*,\lp T_*)$ that produces $(x,T)$ and $(x_*,T_*)$, one may formally write the gain part in strong form as follows
\begin{equation}
\label{eq:Qff+}
\begin{split}
Q^+(f,f)d\omega&=\mathbb{E}\left[ \int_{\Omega}f(x,\lp T,t)f(x_*,\lp T_*,t)dxd\lp Tdx_*d \lp T_* \right]\\
&=\mathbb{E}\left[ \int_{\Omega}f(x,\lp T,t)f(x_*,\lp T_*,t)\frac{1}{J} d\omega_* \right] d\omega,
\end{split}
\end{equation}
where $J=\Big|\frac{\partial(T,T_*)}{\partial(\lp T,\lp T_*)}\Big|$ is the Jacobian determinant of the temperature interaction.

Substituting eqs (\ref{eq:Qff-}) and (\ref{eq:Qff+}) into (\ref{eq:Qffsplit}) yields the desired form
\begin{equation}
\label{eq:PT_strong}
\frac{\partial f(x,T,t)}{\partial t}=\mathbb{E}\left[\int_\Omega f(x,\lp T,t)f(x_*,\lp T_*,t)\frac{1}{J} d\omega_*\right]-f(x,T,t)\coloneqq Q(f,f).
\end{equation}
For much of the remaining analysis, we move over to the weak form of eq. (\ref{eq:PT_strong}). We can write the weak form as
\begin{align}
\begin{split}
\label{eq:PTweaklong}
\frac{\partial}{\partial t}\int_{\Omega} f(x,T,t)&\varphi(x,T)d\omega=\int_{\Omega}Q(f,f)\varphi(x,T)d\omega\\
&=\mathbb{E}\left[\frac{1}{2}\int_{\Omega\times\Omega}\big(\varphi(x,T')+\varphi(x_*,T_*')-\varphi(x,T)-\varphi(x_*,T_*)\big) ff_*d\omega d\omega_*\right]
\end{split}\\
\label{eq:PTweakshort}
&=\mathbb{E}\left[\int_{\Omega\times\Omega}\big(\varphi(x,T')-\varphi(x,T)\big) ff_*d\omega d\omega_*\right].
\end{align}
The passage to eq. (\ref{eq:PTweakshort}) is only possible due to the symmetry of the interaction. Swapping $T$ and $T_*$ and $x$ and $x_*$ proves this statement.

\subsubsection{Combining the search dynamic with the temperature exchange}
\label{subsec:population_Lf_Qff}

The kinetic description of the combined processes is simply the sum of both dynamics. We write it in the following way
\begin{equation}
\label{eq:dynamic}
\frac{\partial f}{\partial t}-\mathcal{L}(f)=Q(f,f)
\end{equation}
to make a parallel to classical kinetic theory where we expect a similar form. We notice that our evolution equation is actually a Boltzmann--Povzner equation due to the form of $Q(f,f)$. 
\begin{remark}[Historical detour]
Povzner analyzed this particular form of equation in the 60s as a regularization of the Boltzmann equation by averaging collisions in space^^>\cite{povzner_boltzmann_1962, arkeryd_stationary_1999}. Back then, there was no proof of existence and uniqueness of the classical Boltzmann equation so Povzner simplified things by ``smearing" the process of interactions over the spatial domain. This has the consequence that interactions are no longer confined to single points but can occur anywhere, unlike in classical kinetic theory where interactions are inherently local due to the physical interpretation of molecules colliding at specific points in the domain. His results were long ignored by physicists because it held no physical relevance. We now know that his equation is useful when treating stochastic systems of particles that ``communicate" or interact in a different way than by colliding. A general example would be the swarming of animals or the flocking of birds^^>\cite{pareschi_interacting_2014, naldi_particle_2010, albi_binary_2013}. Our operator also ``smears" local interactions over the temperature domain allowing non-local interactions. 
\end{remark}

\section{Analysis and mean-field limit}
\label{sec:analysis}
In this section, we perform some analysis in order to better understand the combined dynamic. We start with the long-time behaviour of the algorithm, prove the decay of first and second moments and then discuss the mean-field limit.

\subsection{Long-time behaviour}
In this subsection, we analyse long-time convergence using an entropy-based argument. An important aspect of the algorithm is temperature decay. We will show that our interaction dynamic indeed drives the average temperature down. Under stronger assumptions, we can also prove the decay of the second moment in the temperatures. 

\subsubsection{Entropy structure of the exploration operator}
\label{subsec:analysis_entropies}

We first recall the standard entropy structure of the Metropolis exploration
operator. For each fixed temperature \(T>0\), the Metropolis kernel satisfies
detailed balance with respect to the Gibbs profile \(\pi_{T,\mathcal F}\) defined
in \eqref{eq:bg_prob}. The corresponding local equilibrium associated with the
joint density \(f(x,T,t)\) is the manifold \(M_{\mathcal F}[f]\) defined
in \eqref{eq:equilibrium}.

For fixed \(T>0\), the Gibbs profile \(\pi_{T,\mathcal F}\) concentrates near the
global minimizers of \(\mathcal F\) as \(T\to0^+\). This follows from the Laplace
principle,
\begin{equation}
\label{eq:laplace}
\lim_{T\rightarrow 0^+}
-T\log\left(
\int_{\mathcal D}g(x)e^{-\mathcal F(x)/T}\,dx
\right)
=
\inf_{x\in\operatorname{supp}(g)}\mathcal F(x),
\end{equation}
for suitable continuous probability densities \(g\). This observation explains
the classical simulated annealing mechanism: relaxation toward Gibbs profiles,
combined with cooling, leads to concentration near global minimizers.

The following entropy functional is the temperature average of the fixed-temperature
Shannon--Boltzmann entropy
\begin{equation}
\label{eq:entr}
H(f|M_{\mathcal F})
:=
\int_{\Omega}
f(x,T,t)\log\left(
\frac{f(x,T,t)}{M_{\mathcal F}(x,T,t)}.
\right)\,d\omega,
\end{equation}

Setting
\[
\varphi(x,T)
=
\log\left(
\frac{f(x,T,t)}{M_{\mathcal F}(x,T,t)}
\right)
\]
in \eqref{eq:SAweak}, and using the fact that \(F(T,t)\) is not changed by
\(\mathcal L\), one obtains the entropy dissipation identity
\begin{equation}
\label{eq:SAentr_long}
\begin{split}
\int_{\Omega}\mathcal L(f)
\log\left(
\frac{f(x,T,t)}{M_{\mathcal F}(x,T,t)}
\right)d\omega
=
-\frac12\mathbb E\bigg[
\int_{\Omega}
&\mathcal B_{\mathcal F}(x\rightarrow x',T)
M_{\mathcal F}(x,T,t)
\\
&\times
h\left(
\frac{f(x',T,t)}{M_{\mathcal F}(x',T,t)},
\frac{f(x,T,t)}{M_{\mathcal F}(x,T,t)}
\right)d\omega
\bigg],
\end{split}
\end{equation}
where
\[
h(a,b)=(a-b)(\log a-\log b)\ge0.
\]
We therefore have the following entropy decay property for the exploration
operator.

\begin{theorem}[Entropy decay for the exploration operator]
Along sufficiently regular solutions of
\[
\partial_t f=\mathcal L(f),
\]
one has
\begin{equation}
\label{eq:SAentr_short}
\frac{d}{dt}H(f|M_{\mathcal F})
=
-I_{\mathcal F}[f]\le0,
\end{equation}
where
\begin{equation}
\label{eq:If}
I_{\mathcal F}[f]
=
\frac12\mathbb E\left[
\int_{\Omega}
\mathcal B_{\mathcal F}(x\rightarrow x',T)
M_{\mathcal F}(x,T,t)
h\left(
\frac{f(x',T,t)}{M_{\mathcal F}(x',T,t)},
\frac{f(x,T,t)}{M_{\mathcal F}(x,T,t)}
\right)d\omega
\right].
\end{equation}
\end{theorem}

If, in addition, an entropy-production inequality
\begin{equation}
\label{eq:entprod}
I_{\mathcal F}[f]\geq \tau H(f|M_{\mathcal F})
\end{equation}
holds for some \(\tau>0\), then
\begin{equation}
\label{eq:H_const}
H(f(t)|M_{\mathcal F}(t))
\le
H(f_0|M_{\mathcal F}(0))e^{-\tau t}.
\end{equation}
By the Csiszár--Kullback inequality
\cite{csiszar_information-type_1967,kullback_lower_1967}, this implies convergence
of \(f\) toward the local equilibrium \(M_{\mathcal F}(0)=F(T,0)\pi_T(x)\) in \(L^1(\Omega)\), for
the exploration dynamics alone.

The entropy-production inequality \eqref{eq:entprod} is an additional coercivity assumption on the Metropolis kernel, analogous to a spectral-gap or logarithmic-Sobolev estimate. It is not asserted here in full generality.
In the full CAST dynamics, the interaction operator \(Q(f,f)\) changes the
temperature marginal \(F(T,t)\), and no global Gibbs equilibrium in the joint
variables \((x,T)\) is claimed.
\begin{remark}[Time dependent cooling]
Classical simulated annealing employs a continuous time dependent cooling scheme. This means the Shannon--Boltzmann entropy has some additional term that must be bounded. We refer to^^>\cite{pareschi_optimization_2024} for the full details but the main idea is that eq. (\ref{eq:H_const}) now reads
\begin{equation*}
\frac{d}{dt}H(f|M_\mathcal{F}) \leq -\tau H(f|M_\mathcal{F})-\frac{T'(t)}{T^2(t)}||\mathcal{F}||_{L_\infty(\mathbb{R}^d)}.
\end{equation*}
In order for the inequality to hold, we must ensure that $\frac{T'(t)}{T^2(t)}$ decays with time, i.e., 
\begin{equation*}
\lim_{t\rightarrow\infty}\frac{T'(t)}{T^2(t)}=0.
\end{equation*} 
If, for example, we choose $T(t)=\frac{1}{t}$ then 
\begin{equation}
\lim_{t\rightarrow\infty}\frac{T'(t)}{T^2(t)}=\lim_{t\rightarrow\infty}\frac{\left(\frac{1}{t}\right)'}{\left(\frac{1}{t}\right)^2}=\lim_{t\rightarrow\infty}\frac{-\frac{1}{t^2}}{\frac{1}{t^2}}=-1\neq0,
\end{equation}
violating the aforementioned condition to prove entropy decay. If instead we use eq. (\ref{eq:temp}) in continuous (and slightly different) form then
\begin{equation}
\lim_{t\rightarrow\infty}\frac{T'(t)}{T^2(t)}=\frac{\left(\frac{1}{\log(t)}\right)'}{\left(\frac{1}{\log(t)}\right)^2}=\frac{-\frac{1}{t\log(t)^2}}{\frac{1}{\log(t)^2}}=\lim_{t\rightarrow\infty}-\frac{1}{t}=0.
\end{equation}
This recalls the classical role of inverse-logarithmic cooling in simulated
annealing. In CAST, however, the temperature is not externally prescribed but
evolves through the interaction operator \(Q(f,f)\).
\end{remark}

\subsubsection{Evolution of moments in temperature}
\label{subsec:analysis_moments}
The purpose of the temperature interaction is to generate a collective cooling
mechanism. We show below that, when \(\mu<\lambda\), the expected mean
temperature is non-increasing. Under stronger assumptions, one can also obtain
dissipative contributions in the evolution of the second moment. 

Since the exploration operator $\mathcal L$ leaves the temperature variable unchanged, we can analyse the behaviour of temperature mean and second moment by looking at the sign of equation (\ref{eq:PTweakshort}) when 
\begin{equation}
\varphi(x,T) = T \text{ or } \varphi(x,T) = T^2.
\end{equation}
The sign of the moment equations is governed by 
\begin{equation*}
\mathbb{E}\left[\varphi(x,T')-\varphi(x,T)\right]
\end{equation*}
inside the double integral since $f(x,T,t)f(x_*,T_*,t)$ is always positive. The following two lemmas discuss the behavior of these moment equations.

\begin{lemma}[Evolution of the first moment]
\label{lemma:moment1}
The evolution of $M_1:=\int_{\Omega}fTd\omega$ for $\mu<\lambda$ satisfies
\begin{equation*}
\frac{dM_1}{dt} = (\mu-\lambda)\int_{\Omega\times\Omega}(T-T_*)\chi_\mathcal{F}(x,x_*,T,T_*){ff_*}d\omega d\omega_* \leq 0.
\end{equation*}
\end{lemma}
\begin{proof}
We write
\begin{equation*}
\mathbb{E}\left[ T'-T\right] =  -\lambda(T-T_*)\chi_\mathcal{F}(x,x_*,T,T_*)-\mu(T-T_*)\chi_\mathcal{F}(x_*,x,T_*,T)
\end{equation*}
since $\mathbb{E}\left[\zeta\right] = 0$. Substituting inside the weak form yields
\begin{align*}
\frac{dM_1}{dt} &=\int_{\Omega\times\Omega}\left[-\lambda(T-T_*)\chi_\mathcal{F}(x,x_*,T,T_*)-\mu(T-T_*)\chi_\mathcal{F}(x_*,x,T_*,T)\right]ff_*d\omega d\omega_*\\
&=(\mu-\lambda)\int_{\Omega\times\Omega}(T-T_*)\chi_\mathcal{F}(x,x_*,T,T_*){ff_*}d\omega d\omega_*.
\end{align*}
Since \(\chi_\mathcal{F}(x,x_*,T,T_*)\neq0\) implies \(T>T_*\), the integral on
the right-hand side is nonnegative. Therefore, if \(\mu<\lambda\), the expected
mean temperature is non-increasing. 
\end{proof}
Notice that no closed expression in terms of
the temperature marginal alone is obtained, since the interaction indicator also
depends on the ordering of the objective values \(\mathcal F(x)\) and
\(\mathcal F(x_*)\).
\begin{remark}
\label{rem:speed} This result is consistent with Remark \ref{ex:viewpoint} where we claimed that for $\lambda=\mu$ the mean temperatures would be conserved. Indeed, from the proof of Lemma \ref{lemma:moment1} we get
\begin{equation*}
\frac{dM_1}{dt}=0,
\end{equation*}
if $\mu=\lambda$. For $\mu<\lambda$, we see that the particles lose temperature over time at a rate loosely proportional to $\lambda-\mu$.
\end{remark}

\begin{lemma}[Evolution of the second moment]
\label{lemma:moment2}
The evolution of $
M_2:=\int_{\Omega}fT^2d\omega$
satisfies
 \begin{equation*}
\frac{dM_2}{dt}=\int_{\Omega\times\Omega}\left[(\mu-\lambda)(T^2-T_*^2)+(\lambda^2+\mu^2-\lambda-\mu)(T-T_*)^2\right]\chi_\mathcal{F}(x,x_*,T,T_*)ff_*d\omega d\omega_*.
\end{equation*}
for $\lambda,\mu\in [0,1]$, $\mu\leq\lambda$ and $\sigma(x,x_*,T,T_*)=0$. 
\end{lemma}

\begin{proof}
We start from the post–interaction temperature update
\[
T' = T - \lambda (T - T_*)\chi_\mathcal{F}(x,x_*,T,T_*) - \mu (T - T_*)\chi_\mathcal{F}(x_*,x,T_*,T),
\]
and compute
\[
\begin{split}
\mathbb{E}[T'^2 - T^2]
= &-2\lambda T(T-T_*)\chi_\mathcal{F}(x,x_*,T,T_*) 
  -2\mu T(T-T_*)\chi_\mathcal{F}(x_*,x,T_*,T)\\
  &+ \lambda^2 (T-T_*)^2\chi_\mathcal{F}(x,x_*,T,T_*) 
  + \mu^2 (T-T_*)^2\chi_\mathcal{F}(x_*,x,T_*,T),
  \end{split}
\]
where we used  
\[
\chi_\mathcal{F}(x,x_*,T,T_*)\chi_\mathcal{F}(x_*,x,T_*,T)=0,\qquad \chi^2_\mathcal{F}(x,x_*,T,T_*)=\chi_\mathcal{F}(x,x_*,T,T_*).
\]
Integrating against $ff_*$ on $\Omega\times\Omega$ yields
\begin{equation*}
\begin{split}
\frac{dM_2}{dt}&=\int_{\Omega\times\Omega}\left[-2\lambda T(T-T_*)-2\mu T_*(T_*-T)\right]\chi_\mathcal{F}(x,x_*,T,T_*)ff_*d\omega d\omega_*\\
&+\int_{\Omega\times\Omega}(\lambda^2+\mu^2)(T-T_*)^2\chi_\mathcal{F}(x,x_*,T,T_*)ff_*d\omega d\omega_*.
\end{split}
\end{equation*}
Expanding the first integral and rearranging yields
 \begin{equation*}
\begin{split}
\frac{dM_2}{dt}&=\int_{\Omega\times\Omega}(\mu-\lambda)(T^2-T_*^2)\chi_\mathcal{F}(x,x_*,T,T_*)ff_*d\omega d\omega_*\\
&+\int_{\Omega\times\Omega}(\lambda^2+\mu^2-\lambda-\mu)(T-T_*)^2\chi_\mathcal{F}(x,x_*,T,T_*)ff_*d\omega d\omega_*.
\end{split}
\end{equation*}

The first integral is either negative or zero if $\mu\leq\lambda$. Since $\lambda(\lambda-1)+\mu(\mu-1)\leq 0$ for $\lambda,\mu\in[0,1]$, the second integral is non-positive, and it is strictly negative if the support of
$\chi_\mathcal{F}(x,x_*,T,T_*)$ has positive measure and $\lambda,\mu\in(0,1)$.
Thus, in the absence of interaction noise, \(M_2\) is non-increasing under the
stated assumptions. 
\end{proof}
As in the first-moment estimate, we do not obtain a closed
formula depending only on the temperature marginal, because the interaction
indicator couples temperature and position through the ordering of
\(\mathcal F(x)\) and \(\mathcal F(x_*)\).

\begin{remark}[Effect of noise]
\label{rem:lemma2}
If the stochastic term is included, the right-hand side in Lemma \ref{lemma:moment2} is supplemented by the nonnegative contribution
\[
\mathbb{E}[\zeta^2]\int_{\Omega\times\Omega} \sigma^2(x,x_*,T,T_*)\,\,f f_*\,d\omega\,d\omega_*.
\]
Thus, in the noisy case, monotone decay of \(M_2\) is not automatic.
It holds only if this positive contribution is dominated by the
dissipative part of the deterministic interaction. In particular, the noise-free estimate should be understood as the baseline coercive mechanism, while bounded interaction noise may be allowed under a suitable smallness condition.
\end{remark}

\subsection{Mean-field approximation}
\label{sec:meanfield}
Due to the coupling between position and temperature, the full kinetic equation is
not closed at the level of either marginal alone. In addition to the marginal $F(T,t)$ defined 
in \eqref{eq:tmarginal}, we therefore introduce the
additional marginal
\begin{equation}
\rho(x,t):=\int_{\mathbb R^+}f(x,T,t)\,dT,
\end{equation}
and derive formal reduced descriptions, emphasizing where closure is lost. 

\subsubsection{Temperature averaged case}
\label{subsec:meanfield_temp}
For the temperature averaged case, we can write
\begin{equation}
\frac{\partial}{\partial t}\int_{\mathbb{R}^+}f(x,T,t)dT-\int_{\mathbb{R}^+}\mathcal{L}(f)dT=\int_{\mathbb{R}^+}Q(f,f)dT.
\end{equation}
The temperature interaction does not change the position variable. This can be
seen directly from the weak form of \(Q\). Indeed, choosing a test function
\(\varphi(x,T)=\psi(x)\), independent of \(T\), in \eqref{eq:PTweaklong}, gives
\[
\begin{aligned}
\int_\Omega Q(f,f)\psi(x)\,d\omega
&=
\frac12\mathbb E
\int_{\Omega\times\Omega}
\big[
\psi(x)+\psi(x_*)-\psi(x)-\psi(x_*)
\big]
ff_*\,d\omega\,d\omega_* \\
&=0.
\end{aligned}
\]
Therefore, 
\[
\int_{\mathbb R^+}Q(f,f)\,dT=0,
\]
and hence

\begin{equation}
\label{eq:rho}
\frac{\partial \rho(x,t)}{\partial t}-\int_{\mathbb{R}^+}\mathcal{L}(f)dT=0.
\end{equation}
In general, the above equation does not admit a closed form in terms of the marginal density $\rho(x,t)$. Even in the mean-field scaling where the operator $\mathcal{L}(f)$ reduces to a Langevin dynamics (see^^>\cite{pareschi_optimization_2024} for more details), we will get 
\[
\frac{\partial \rho(x,t)}{\partial t} = \nabla_x \cdot \left(\rho(x,t)\nabla_x \mathcal{F}(x) \right) + \int_{\mathbb{R}^+} T\Delta_{xx} f(x,T,t)\,dT.
\]
We now apply integration by parts and define the average temperature at point \(x\) as
\begin{equation}
\bar{T}(x,t) := \frac{1}{\rho(x,t)} \int T f(x,T,t) \, dT = \frac{M_1(x,t)}{\rho(x,t)}.
\end{equation}
This gives
\begin{equation}
\partial_t\rho(x,t)
=
\nabla_x\cdot\big(\rho(x,t)\nabla_x\mathcal F(x)\big)
+
\Delta_x M_1(x,t),
\end{equation}
or equivalently, since \(M_1(x,t)=\rho(x,t)\bar T(x,t)\), one obtains
\begin{equation}
\partial_t\rho
=
\nabla_x\cdot\left(
\rho\nabla_x\mathcal F
+
\bar T\nabla_x\rho
+
\rho\nabla_x\bar T
\right).
\end{equation}

Thus, the spatial marginal equation is not closed unless an additional closure is
introduced. If the local average temperature is approximately homogeneous in
space, namely
\[
\bar T(x,t)\approx \bar T(t),
\]
then one obtains the closed nonlinear Fokker--Planck approximation
\begin{equation}
\partial_t\rho
=
\nabla_x\cdot\left(
\rho\nabla_x\mathcal F
+
\bar T(t)\nabla_x\rho
\right).
\end{equation}
This is the spatial marginal model corresponding to Langevin exploration with an
effective time-dependent temperature generated by the collective dynamics.
Note that, the above Langevin approximation for the exploration operator applies to finite-variance proposals, such as Gaussian proposals, whereas it is not defined for Cauchy proposals.

\subsubsection{Space averaged case}
\label{subsec:meanfield_space}
For the space averaged case, we can write
\begin{equation}
\frac{\partial}{\partial t}\int_{\mathcal D}f(x,T,t)dx-\int_{\mathcal D}\mathcal{L}(f)dx=\int_{\mathcal D}Q(f,f)dx.
\end{equation}
Expanding the second term yields
\begin{align}
\begin{split}
\int_{\mathcal D}\mathcal{L}(f)dx&=\int_{\mathcal D}\mathbb{E}\left[\mathcal{B}_\mathcal{F}(x'\rightarrow x)f(x',T,t)-\mathcal{B}_\mathcal{F}(x\rightarrow x')f(x,T,t)\right] dx\\
&=\int_{\mathcal D}\mathbb{E}\left[\mathcal{B}_\mathcal{F}(x'\rightarrow x)f(x',T,t)\right] dx-\int_{\mathcal D}\mathbb{E}\left[\mathcal{B}_\mathcal{F}(x\rightarrow x')f(x,T,t)\right] dx\\
&=\int_{\mathcal D}\mathbb{E}\left[\mathcal{B}_\mathcal{F}(x\rightarrow x')f(x,T,t)\right] dx-\int_{\mathcal D}\mathbb{E}\left[\mathcal{B}_\mathcal{F}(x\rightarrow x')f(x,T,t)\right] dx\\
&=0
\end{split}
\end{align}
by Lemma \ref{lemma:symm}. Thus, we get
\begin{equation}
\label{eq:tau}
\frac{\partial F(T,t)}{\partial t}=\int_{\mathcal D}Q(f,f)dx.
\end{equation}
Unfortunately, it is impossible to write the right-hand side as $Q(F,F)$ because of the implicit dependence of $x$ and $x_*$ in $T'$ and $T_*'$. The weak form is therefore similar to eq. (\ref{eq:PTweakshort}) (unlike before, we do not use the $\Omega$ notation to underline the integration over the spatial domain), i.e., 
\begin{equation}
\label{eq:tau_weak}
\frac{\partial}{\partial t}\int_{\mathbb{R}^+} F(T,t)\varphi(T)dT=\mathbb{E}\left[\int_{\mathbb{R}^+\times\mathbb{R}^+}\left[\int_{\mathcal D\times\mathcal D}\big(\varphi(T')-\varphi(T)\big) ff_*dxdx_*\right] dTdT_*\right].
\end{equation}

\subsubsection{Mean-field scaling of the temperature exchange}
We consider the following scaling
\begin{equation*}
t\rightarrow t/\varepsilon,\quad \lambda\rightarrow \varepsilon \lambda,\quad \mu\rightarrow \varepsilon \mu, \quad \sigma \rightarrow \sqrt{\varepsilon} \sigma.
\end{equation*}
Therefore we have
\[
T'-T
=
-\varepsilon\lambda(T-T_*)\chi_\mathcal{F}(x,x_*,T,T_*)
-\varepsilon\mu(T-T_*)\chi_\mathcal{F}(x_*,x,T_*,T)
+\sqrt{\varepsilon}\sigma(x,x_*,T,T_*)\zeta.
\]
For $\varepsilon \ll 1$, we can write the Taylor expansion
\begin{equation}
\varphi(T')=\varphi(T)+(T'-T)\frac{d\varphi(T)}{dT}+\frac{1}{2}(T'-T)^2\frac{d^2\varphi(T)}{dT^2}+O\left(\varepsilon^{3/2}\right).
\end{equation} 
Substituting into eq. (\ref{eq:tau_weak}) and splitting the integral yields
\begin{align}
\frac{\partial}{\partial t}\int_{\mathbb{R}^+} F(T,t)\varphi(T)dT= &\underbrace{\frac1{\varepsilon}\int_{\mathbb{R}^+\times\mathbb{R}^+}\int_{\mathcal D\times\mathcal D}\mathbb{E}\left[T'-T\right]\frac{d\varphi(T)}{dT}ff_*dxdx_*dTdT_*}_{:=I_1} \\
+\frac{1}{2}&\underbrace{\frac1{\varepsilon}\int_{\mathbb{R}^+\times\mathbb{R}^+}\int_{\mathcal D\times\mathcal D}\mathbb{E}\left[(T'-T)^2\right]\frac{d^2\varphi(T)}{dT^2} ff_*dxdx_* dTdT_*}_{:=I_2}.
\end{align}
Expanding \(I_1\) under the scaling gives
\begin{equation}
\begin{split}
I_1
=
\int_{\mathbb R^+\times\mathbb R^+}
\int_{\mathcal D\times\mathcal D}
&\left[
-\lambda(T-T_*)\chi_\mathcal F(x,x_*,T,T_*)
\right.\\
&\left.
-\mu(T-T_*)\chi_\mathcal F(x_*,x,T_*,T)
\right]
\frac{d\varphi(T)}{dT}
ff_*\,dx\,dx_*\,dT\,dT_* .
\end{split}
\end{equation}
The diffusion contribution is
\begin{equation}
I_2
=
\mathbb E[\zeta^2]
\int_{\mathbb R^+\times\mathbb R^+}
\int_{\mathcal D\times\mathcal D}
\sigma(x,x_*,T,T_*)^2
\frac{d^2\varphi(T)}{dT^2}
ff_*\,dx\,dx_*\,dT\,dT_* .
\end{equation}
At this level the equation for the temperature marginal is not closed, because
the drift depends on the correlation between temperature and objective value. To
obtain a reduced model, we introduce the following quasi-equilibrium closure with
respect to the exploration operator:
\begin{equation}
\label{eq:closure_gibbs}
f(x,T,t)\approx F(T,t)\pi_T(x),
\end{equation}
where \(F(T,t)=\int_{\mathcal D}f(x,T,t)\,dx\) and \(\pi_T\) is defined in
\eqref{eq:bg_prob}. This closure is natural when the Metropolis exploration
relaxes faster than the temperature-exchange dynamics, so that the conditional
spatial distribution at fixed temperature is close to its Gibbs profile.

Under this closure, define the ordering kernel
\begin{equation}
\label{eq:ordering_kernel}
K(T,T_*)
:=
\int_{\mathcal D\times\mathcal D}
\Psi\big(\mathcal F(x)<\mathcal F(x_*)\big)
\pi_T(x)\pi_{T_*}(x_*)\,dx\,dx_* .
\end{equation}
This kernel measures, under Gibbs distributions at temperatures \(T\) and
\(T_*\), the probability that a particle at temperature \(T\) has a lower
objective value than a particle at temperature \(T_*\). In general
\(K(T,T_*)\neq 1/2\), which reflects the coupling between temperature and
objective value.

The closed drift for the temperature marginal is then
\begin{equation}
\label{eq:closed_drift}
\mathcal A[F](T)
=
-\lambda F(T,t)\int_0^T
(T-T_*)K(T,T_*)F(T_*,t)\,dT_*
+
\mu F(T,t)\int_T^\infty
(T_*-T)K(T_*,T)F(T_*,t)\,dT_* .
\end{equation}
Similarly, the closed diffusion coefficient is
\begin{equation}
\label{eq:closed_diffusion}
\mathcal D[F](T)
=
F(T,t)
\int_0^\infty\left[
\int_{\mathcal D\times\mathcal D}
\sigma(x,x_*,T,T_*)^2
\pi_T(x)\pi_{T_*}(x_*)\,dx\,dx_* \right]
F(T_*,t)\,dT_*,
\end{equation}
Therefore, under the quasi-equilibrium closure \eqref{eq:closure_gibbs}, the
grazing-interaction limit yields the closed Fokker--Planck equation
\begin{equation}
\label{eq:FP_temperature_closed}
\partial_t F(T,t)
=
-\partial_T\mathcal A[F](T)
+
\frac{\mathbb E[\zeta^2]}{2}
\partial_{TT}\mathcal D[F](T).
\end{equation}
Equation \eqref{eq:FP_temperature_closed} is understood formally on
\(\mathbb{R}_+\), with zero-flux behaviour at \(T=0\) and sufficient decay as \(T\to\infty\). More precisely, writing
\begin{equation*}
\mathcal J_F(T,t)=\mathcal{A}[F](T)-\frac{\mathbb{E}[\zeta^2]}{2}\partial_TD[F](T),
\end{equation*}
we require
\begin{equation*}
\mathcal J_F(0,t)=0,
\qquad
\lim_{T\to\infty}\mathcal J_F(T,t)=0,
\end{equation*}
so that the total mass of \(F\) is conserved. In the multiplicative-noise case \(\sigma(x,x_*,T,T_*)=T\), one has the inner integral of (\ref{eq:closed_diffusion}) $=T^2$. If the total mass is normalized to one, then
\[
\mathcal D[F](T)=T^2F(T,t).
\]

\section{Implementation aspects}
\label{sec:impl}
In this section, we give implementation details, generalizations and parameter choices used in the experiments to follow.

\subsection{Implementation}

The numerical implementation of our algorithm relies on the standard Metropolis algorithm combined with a direct simulation Monte Carlo method for the Boltzmann interactions.\cite{babovsky_simulation_1986,borghi_wasserstein_2026} We start with some notation:
\begin{itemize}
\item We define the $j^{th}$ dimension of particle $i$ at time step $n$ as $X^{n,i}_j$. The matrix of $N$ particles with dimension $d$ at time step $n$ is boldfaced as $\mathbf{X}^{n}$.
\item We define the best particle position $\widehat{X}^n$ (with corresponding temperature $\widehat{T}^n$) at time step $n$ as
\begin{equation}
\label{eq:best}
i_n^*\in\argmin_{1\le i\le N}\mathcal F(X^{n,i}),
\qquad
\widehat X^n:=X^{n,i_n^*},
\qquad
\widehat T^n:=T^{n,i_n^*}.
\end{equation}
\item We define the average particle position $\widetilde{X}^n$ at time step $n$ as 
\begin{equation}
\label{eq:avg}
Adm^n
=
\left\{
X^{n,i}:\ \|X^{n,i}\|_\infty\leq1,\ i=1,\ldots,N
\right\},
\qquad
\widetilde X^n
=
\frac{1}{|Adm^n|}
\sum_{X\in Adm^n}X .
\end{equation}
Here \(Adm^n\) is the set of particles that remain inside the initial computational
box. This restricted average is used only as a diagnostic, since rare particles
that travel far outside the initial box may otherwise dominate the empirical mean.
\item The error is measured using the mean squared error and is computed as
\begin{equation}
\text{MSE}(Y) = \frac{1}{d}\sum_{j=1}^d\left(Y_j-x^*_j \right)^2,
\end{equation}
where $Y$ is some vector $\in\mathbb{R}^d$ and $x^*$ is the global minimum. The subscript denotes the $j^{th}$ dimension.
\item In order to reduce Monte Carlo fluctuations, we run the optimization algorithms
over \(N_{\rm runs}\) independent runs and compute empirical averages. For a
quantity of interest \(Z\), we use
\begin{equation}
\mathbb E[Z]\approx
\frac{1}{N_{\rm runs}}
\sum_{n=1}^{N_{\rm runs}}Z^{(n)},
\end{equation}
where \(Z^{(n)}\) denotes the value of \(Z\) in run \(n\).
\end{itemize} 
The algorithm now reads,

\begin{algorithm}[htb]
\caption{Collective Annealing by Switching Temperatures}\label{alg:nb}
\begin{algorithmic}[1]
\Require Vectors $\mathbf{X^0}\in\mathbb{R}^N$ and $\mathbf{T^0}\in\mathbb{R}^N$
\For{$n=1$ to $n_{\text{steps}}$}
\State $\widetilde{\mathbf{X}}^{n+1}\gets \mathbf{X}^n+\eta(\mathbf{T})\mathbf{\xi}$ \Comment{$\mathbf{\xi}\in\mathbb{R}^N\sim p(\xi,T)$}
\State $\mathbf{X}^{n+1} \gets \widetilde{\mathbf{X}}^{n+1}_{\text{accepted}}$ \Comment{Set accepted particles according to eq.(\ref{eq:SAacceptreject})}
\State $N_c \gets \text{Iround}(N/2)$
\State select $N_c$ pairs $(i,j)$ \Comment{Uniformly among all possible pairs without replacement}
\For {each pair $(i,j)$}
\State Compute $T_i'$ and $T_j'$ \Comment{According to eq. (\ref{eq:PT_trans})}
\State $T_i^{n+1} \gets T_i'$, $T_j^{n+1} \gets T_j'$
\EndFor
\State $T_i^{n+1} \gets T_i^n$ for all unselected particles.
\EndFor
\end{algorithmic}
\end{algorithm}
where the $\text{Iround(x)}$ function denotes a stochastic integer rounding, i.e.,
\begin{equation*}
\text{Iround}(x)=\begin{dcases}
\text{int}(x)+1 & \text{with prob.}\quad x-\text{int}(x)\\
\text{int}(x) & \text{with prob.}\quad \text{int}(x)+1-x\\
\end{dcases} .
\end{equation*}
$\mathbf{T_0}$ is a vector $\in\mathbb{R}^N$ containing the initial temperature for the $N$ particle system. For the remainder of the tests, the initial temperatures are \textbf{uniformly} taken between $T_l$ and $T_h$ where $0<T_l<T_h$. 

\subsection{Noise term choice}
\label{sec:noise}
In the numerical tests, we will take the noise term as
\begin{equation}
\label{eq:noise}
\sigma(x,x_*,T,T_*)=T.
\end{equation}
This is an intuitive choice since the added noise in the interaction is directly
proportional to the current temperature. According to \eqref{eq:pos}, the noise
must be bounded to ensure positivity of the post-interaction temperatures. With
the choice \(\sigma(x,x_*,T,T_*)=T\), a simple sufficient condition can be obtained
by considering the worst cases in the temperature update.

For instance, if \(\chi_\mathcal F(x,x_*,T,T_*)=1\), then
\begin{equation}
\begin{split}
T' &= T-\lambda(T-T_*)+T\zeta
    =T(1-\lambda+\zeta)+\lambda T_*,\\
T_*' &= T_*-\mu(T_*-T)+T_*\zeta_*
    =T_*(1-\mu+\zeta_*)+\mu T.
\end{split}
\end{equation}
The worst case corresponds to the smallest admissible noise and to the smallest
value of the opposite temperature. This gives the sufficient conditions
\[
a\le 1-\lambda,
\qquad
a\le 1-\mu.
\]
Since in the simulations we use \(\mu<\lambda\), we impose the stronger condition
\begin{equation}
\label{eq:bound}
a\le 1-\lambda.
\end{equation}
This condition also implies the no-interaction constraint \(a\le1\) associated
with the update \(T'=T+T\zeta\).

It will later be useful to scale the noise down a bit. We therefore introduce the \textbf{noise scaling factor $\bm{\kappa}$} such that the support now reads
\begin{equation}
\label{eq:noise_scaled}
a = \kappa(1-\lambda),
\end{equation}
where $\kappa\leq1$.

Finally, the exploration noise in eq. (\ref{eq:SAexploration}) must be sampled from a distribution we called $p(\xi,T)$. As mentioned in Remark \ref{rem:noise}, we choose Cauchy noise. 

\subsection{Further generalization}

In order to balance the search aspect of the algorithm, controlled by $\mathcal{L}(f)$, and the interaction aspect, controlled by $Q(f,f)$, we introduce \textbf{the interaction strength parameter $\bm{\gamma}$}. The dynamic of eq. (\ref{eq:dynamic}) now reads
\begin{equation}
\label{eq:dynamic_scaled}
\frac{\partial f}{\partial t}-\mathcal{L}(f)=\gamma Q(f,f).
\end{equation}
By decreasing $\gamma$, we allow more exploration and fewer particle-particle interactions.  This can be important in some cases to allow particles to ``settle" a bit before performing a new temperature interaction. Conversely, increasing $\gamma$ allows for more interactions, therefore reducing the temperature faster. In classical kinetic theory, this parameter is closely related to the (inverse) Knudsen number.
\\
On the particle-level itself, i.e., in the Nanbu-Babovsky algorithm above, this is achieved by changing the $5^{th}$ line to 
\begin{equation*}
N_c\leftarrow \text{Iround}(\gamma N/2).
\end{equation*} 
For \(\gamma>1\), pairs are sampled with replacement from all unordered pairs.
Thus a particle may interact more than once during a single time step.
In order to compare our algorithm with classical SA, we want comparable particle temperatures. As a reminder, our algorithm simulates $N$ particles where the highest temperature is given by $T_h\in\mathbb{R}^+$ and the lowest is given by $T_l\in\mathbb{R}^+$. All other $N-2$ temperatures are taken uniformly between $T_h$ and $T_l$. Thus, the average particle temperature is given by 
\begin{equation*}
\bar{T} = \frac{T_h+T_l}{2}.
\end{equation*}
Classical multi-particle SA also simulates $N$ particles but all have the same temperature, we set this to $\widebar{T}$. This gives us the possibility to tweak $T_h$ and $T_l$ while still keeping the value for $\widebar{T}$ fixed. In order to somewhat limit the amount of free parameters, we employ the following logic:
\begin{enumerate}
\item Choose a \textbf{fixed} value for the SA temperature/our average temperature, i.e., $\widebar{T}$.
\item Set $T_h=2\widebar{T}$ and $T_l=0$.
\item Tweak the particle's temperature spread by varying $T_{\mathrm{spread}}$ and setting the new highest and lowest temperatures as
\begin{equation}
\label{eq:tvar}
T_h'=2\widebar{T}-T_{\mathrm{spread}}\quad \text{and} \quad T_l'=T_{\mathrm{spread}}.
\end{equation}
with the restriction that $0<T_{\mathrm{spread}}<\widebar{T}$.
\end{enumerate}

\section{Numerical examples}
\label{sec:num}
In this section, we first introduce the two test functions on which we will test the algorithm, namely Ackley and Rastrigin.\cite{jamil_literature_2013,rastrigin_systems_1974} We then study CAST numerically and compare it with classical SA where the temperature follows either an inverse logarithmic or geometric decay. We perform various experiments in order to understand the effect of the various parameters and to compare performance in low and high dimensions. 
All code used to run the algorithms and generate the figures is available at \begin{center}
\url{https://gitlab.kuleuven.be/numa/public/CAST}
\end{center}

\subsection{Test functions}
The two test functions we will consider were selected because they highlight complementary optimization challenges:
the Ackley function features a moderately multimodal landscape with a smooth basin around the global minimum, useful to assess convergence efficiency and stability, whereas the Rastrigin function exhibits a highly multimodal and oscillatory surface with many regularly spaced local minima, posing a much more demanding test for global exploration and avoidance of local traps.
Together, they form a minimal yet representative benchmark set to evaluate both the exploration and exploitation capabilities of the proposed algorithm.

\begin{figure}[H]
\begin{subfigure}{0.49\textwidth}
\centering
\includegraphics[width=0.9\linewidth]{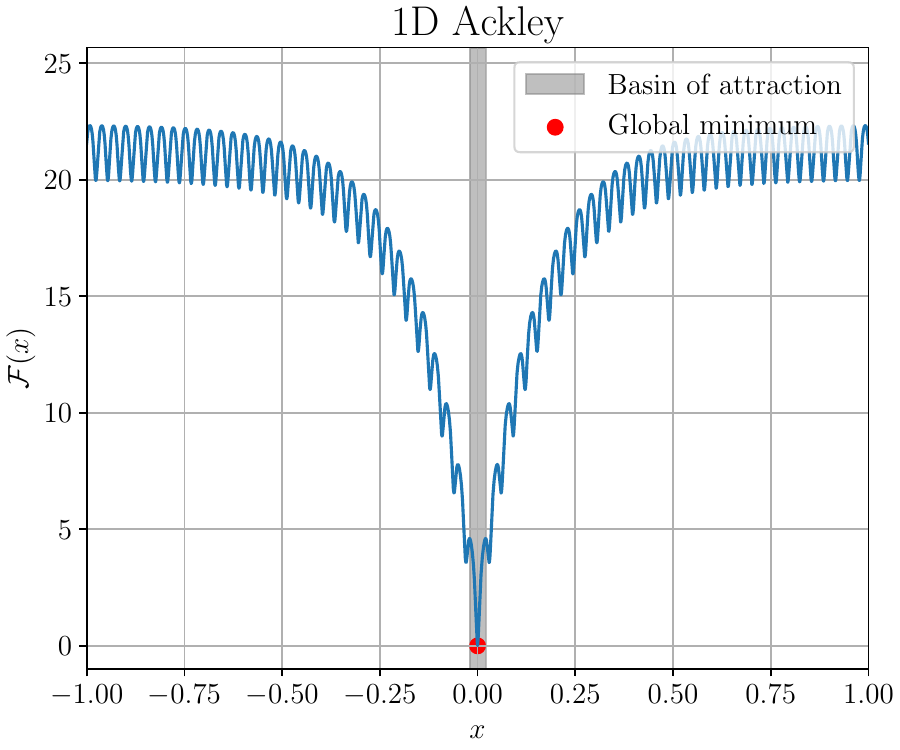}
\caption{Ackley test function.}
\end{subfigure}
\hfill
\begin{subfigure}{0.49\textwidth}
\centering
\includegraphics[width=0.9\linewidth]{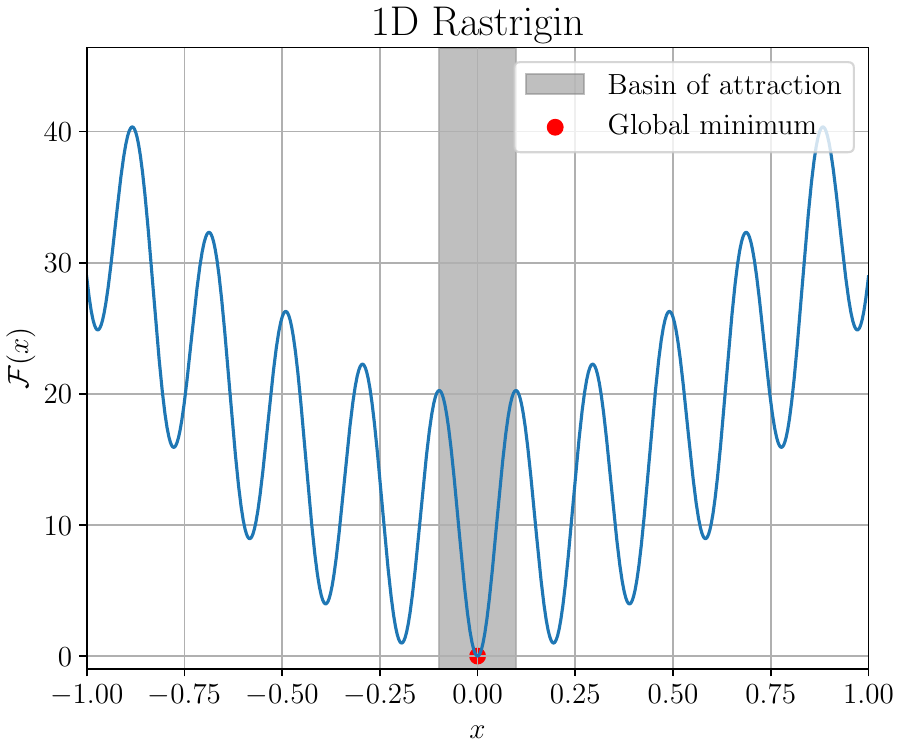}
\caption{Rastrigin test function.}
\end{subfigure}
\caption{Visualisation of the two test functions in dimension one. The shaded area shows the basin of attraction and the red dot indicates the global minimum.}
\end{figure}

\paragraph{Ackley.}
The Ackley function is given by
\begin{equation}
\label{eq:ackley}
\mathcal{F}(x)=-20\exp\Bigg(-\frac{1}{5}\sqrt{\frac{1}{d}\sum_{j=1}^d x_j^2}\Bigg)-\exp\Bigg(\frac{1}{d}\sum_{j=1}^d\cos(2\pi x_j)\Bigg)+20+e.
\end{equation}
where $d$ is the dimension and $e$ is Napier's constant. The function is evaluated on the hypercube $x_j\in[-32.768,32.768]$ for $j=1,\cdots,d$ and its global minimum lies in $x^*=(0,\cdots,0)$ with $\mathcal{F}(x^*)=0$. We call ``the basin of attraction", the area where a particle must simply follow the gradient down to reach the global minimum. For the Ackley function this is given by the open $||\cdot||_\infty$-ball with radius $\approx 0.67$.

\paragraph{Rastrigin.} 
The Rastrigin function is given by
\begin{equation}
\label{eq:rastrigin}
\mathcal{F}(x)
=
10d+\sum_{j=1}^d\left(x_j^2-10\cos(2\pi x_j)\right).
\end{equation}
The function is evaluated on the hypercube $x_j\in[-5.12,5.12]$ for $j=1,\cdots,d$ and its global minimum lies in $x^*=(0,\cdots,0)$ with $\mathcal{F}(x^*)=0$.
The basin for the Rastrigin function is given by the open $||\cdot||_\infty$-ball with radius $0.5$.

Since we are interested in comparing the performance of our algorithm on these two functions, we adopt a rescaling of the domain to the hypercube $x_j\in[-1,1]$ for $j=1,\cdots,d$.

\subsection{Parametric effects on the temperature}

We start our numerical analysis in dimension one. We are initially interested in understanding the effects of the various parameters on the temperature decay. We summarize them here:
\begin{enumerate}
\item The temperature interaction strengths, $\mu<\lambda$ with $\mu,\lambda\in(0,1)$. See eq. (\ref{eq:PT_trans}).

\item The noise scaling, $\kappa\in(0,1)$. See eq. (\ref{eq:noise_scaled}).

\item The interaction strength, $\gamma\in(0,\infty)$. See eq. (\ref{eq:dynamic_scaled}).

\item The temperature spread, $T_{\mathrm{spread}}$. See eq. (\ref{eq:tvar}).

\end{enumerate}
For the following calibration tests, we set $N=4\times 10^4$ and simulate the algorithm over 750 time steps on the one-dimensional Ackley function. Each simulation is repeated and averaged over 100 times to avoid statistical bias. The purpose of these tests is not to perform an exhaustive optimization of the parameters, but rather to verify that the collective temperature interaction can reproduce standard cooling behaviours through suitable choices of the interaction
parameters

\subsubsection{Logarithmic decay}
We wish to tune the parameters such that we \textit{approximate} a logarithmic temperature decay. We fix $\widebar{T}=0.05$ and set the \textit{initial} parameters to 
\begin{equation*}
(\mu,\lambda,\kappa,\gamma,T_{\mathrm{spread}})=(0.2,0.7,1,1,0.02).
\end{equation*}
We will tune the parameters one by one, showing the direct effect of the parameter with a figure and a few key observations. The figures will show: the desired logarithmic decay $\frac{1}{ln(t+e)}$ (black dashed line); the mean temperature (red solid line); and the temperature distribution as a blue gradient color histogram.

\paragraph{Effect of noise scaling parameter $\kappa$}

\begin{figure}[htb]
\centering
\includegraphics[width=0.95\linewidth]{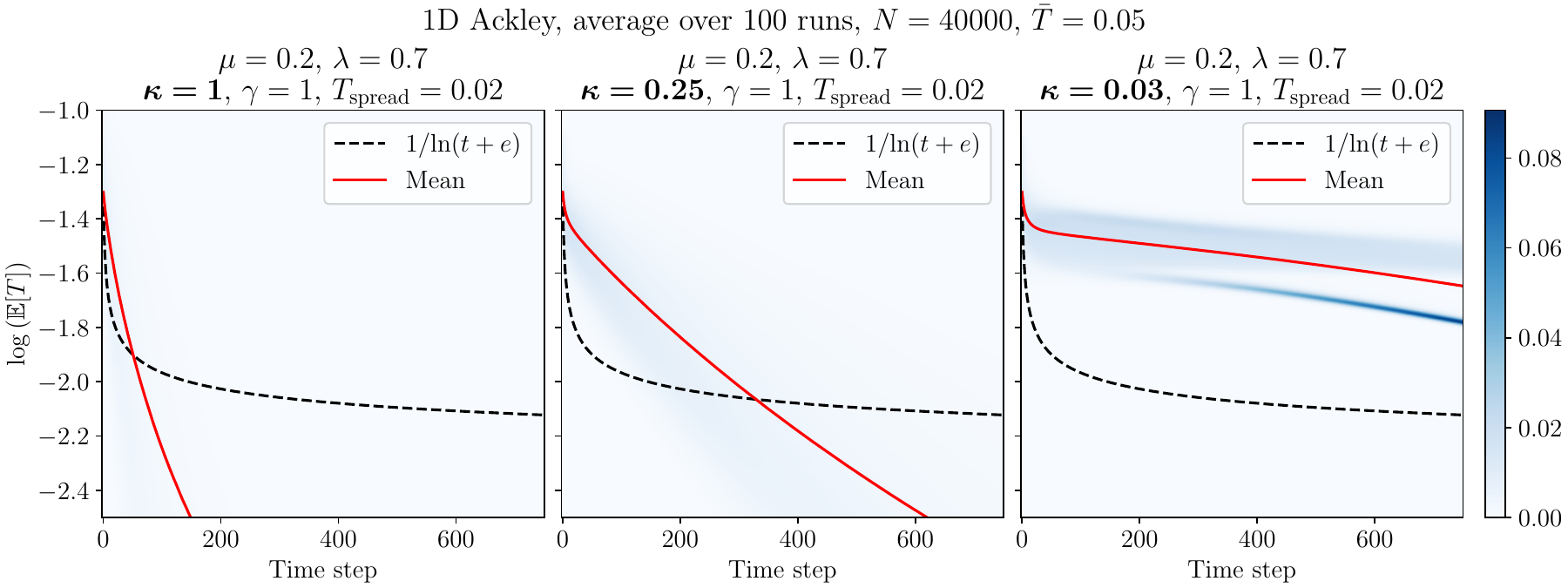}
\caption{Effect of the noise scaling parameter $\kappa$.}
\label{fig:kappa}
\end{figure}

Shown in Figure \ref{fig:kappa}, reducing the support of the random variable $a$ by adjusting $\kappa$ leads us to the following observations: 
\begin{itemize}
\item It is clear that $\kappa=1$ is too large. Too many particle-particle interactions have the effect of reducing the mean temperature too rapidly. 
\item Setting $\kappa=0.7$ reduces the speed at which the temperature decays and also seems to narrow the range of temperatures, as can be see on the histogram.
\item Further reducing $\kappa$ to $0.03$ further slows down the temperature decay. The shape of the mean temperature starts resembling an inverse logarithmic decay.
\end{itemize}

\paragraph{Effect of temperature interaction parameter $\lambda$}
\begin{figure}[htb]
\centering
\includegraphics[width=0.95\linewidth]{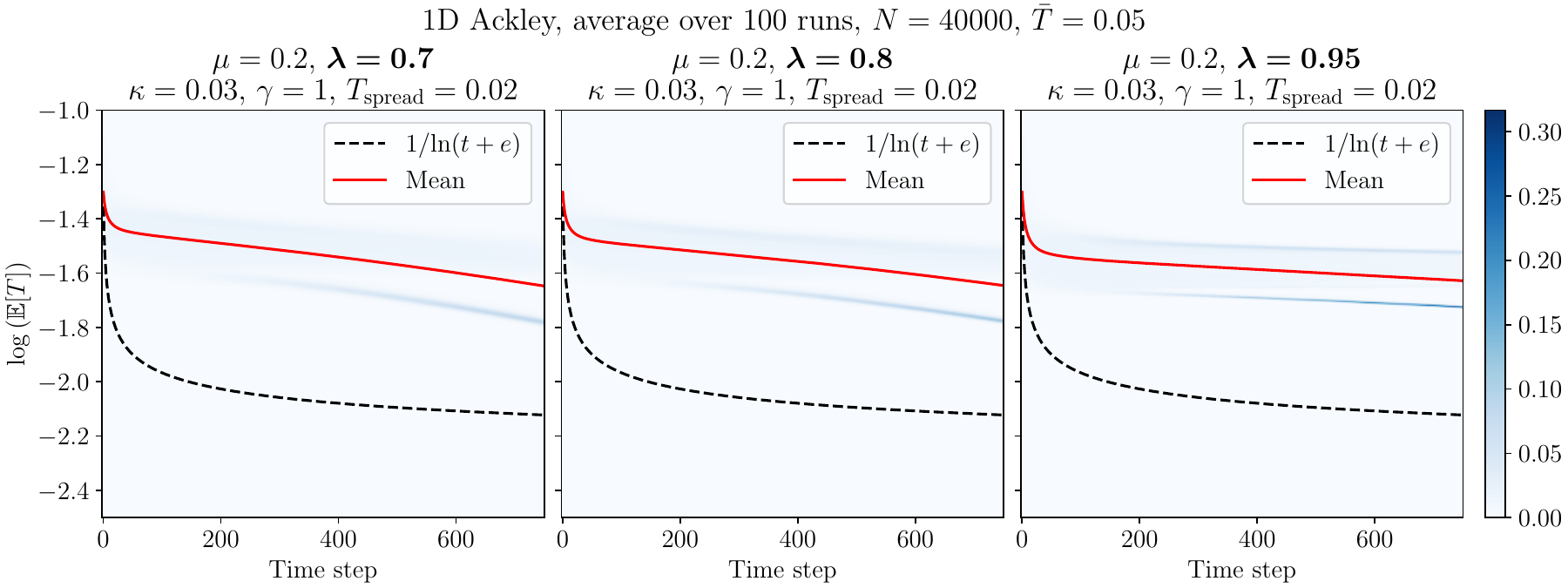}
\caption{Effect of the temperature interaction strength parameter $\lambda$.}
\label{fig:lambda}
\end{figure}

Shown in Figure \ref{fig:lambda}, increasing the temperature interaction strength by tuning $\lambda$ leads us to the following observations:
\begin{itemize}
\item By increasing $\lambda$ up to $0.95$, we observe a slower decay. The mean temperature further approaches the shape of logarithmic decay.
\item We notice that increasing $\lambda$ and decreasing $\kappa$ have similar effects. This is because the support of the noise $a$ is a function of both $\lambda$ and $\kappa$.
\end{itemize}

\paragraph{Effect of interaction strength parameter $\gamma$}

\begin{figure}[htb]
\centering
\includegraphics[width=0.95\linewidth]{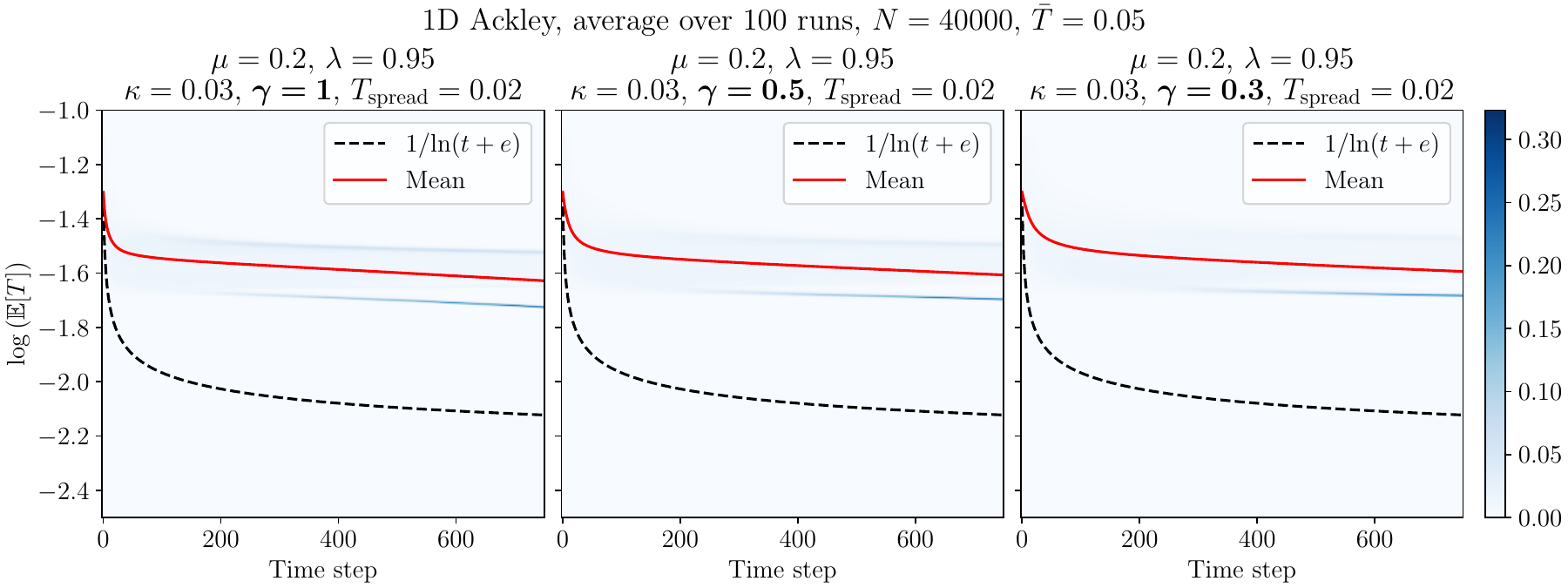}
\caption{Effect of the interaction strength parameter $\gamma$.}
\label{fig:gamma}
\end{figure}

Shown in Figure \ref{fig:gamma}, decreasing the interaction strength $\gamma$ leads us to the following observation:
\begin{itemize}
\item By reducing $\gamma$ to $0.3$, we notice a slight reduction in decay speed as expected by looking at eq. (\ref{eq:dynamic_scaled}).
\end{itemize}

\paragraph{Effect of temperature interaction parameter $\mu$}

\begin{figure}[htb]
\centering
\includegraphics[width=0.95\linewidth]{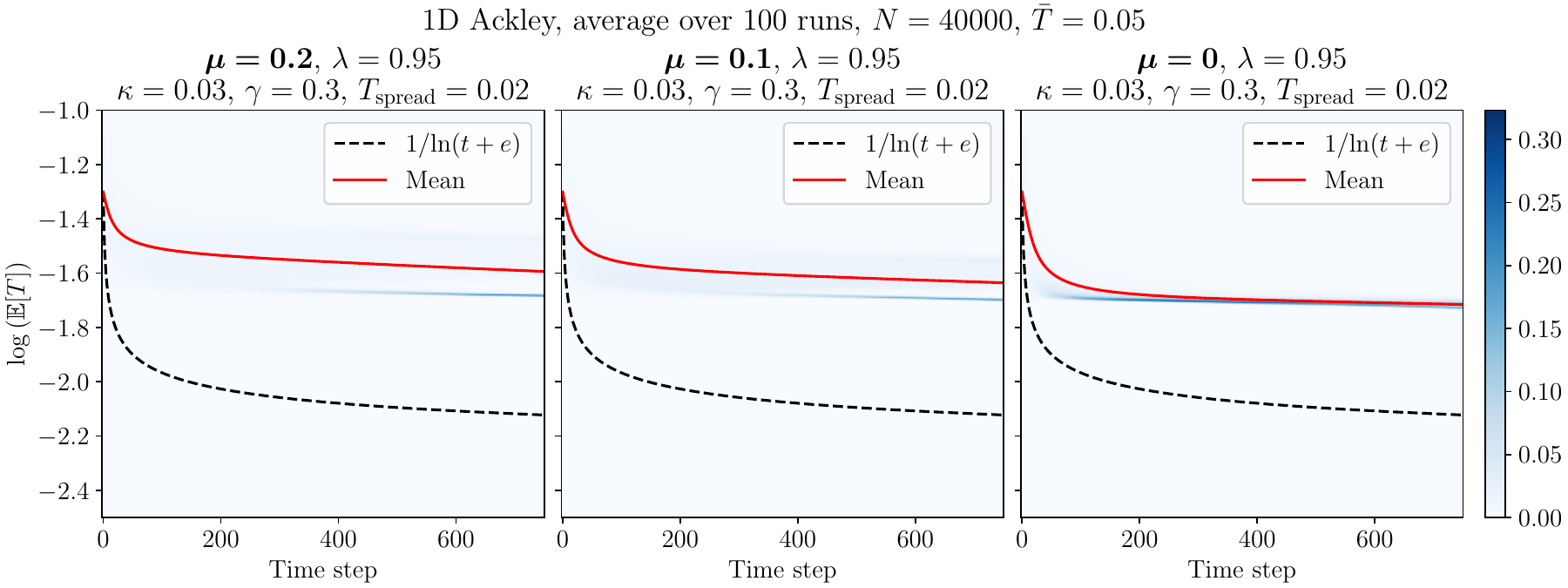}
\caption{Effect of the temperature interaction strength parameter $\mu$.}
\label{fig:mu}
\end{figure}
Shown in Figure \ref{fig:mu}, decreasing the interaction strength parameter $\mu$ leads us to the following observations: 
\begin{itemize}
\item By decreasing $\mu$, we observe an increased temperature decay. This follows from the proof of Lemma \ref{lemma:moment1}. We should indeed expect an increase in decay speed because \(\lambda-\mu\) becomes larger.
\item Looking at the histogram, we see that for small enough $\mu$, all temperatures ``collapse" on the mean. This implies that the spread in the temperatures is decreasing. Although Lemma \ref{lemma:moment2} only proves this in the absence of noise, this empirically shows that the second moment could also decay in some cases (just not uniformly).
\end{itemize}

\paragraph{Effect of temperature spread parameter $T_{\mathrm{spread}}$}

\begin{figure}[htb]
\centering
\includegraphics[width=0.95\linewidth]{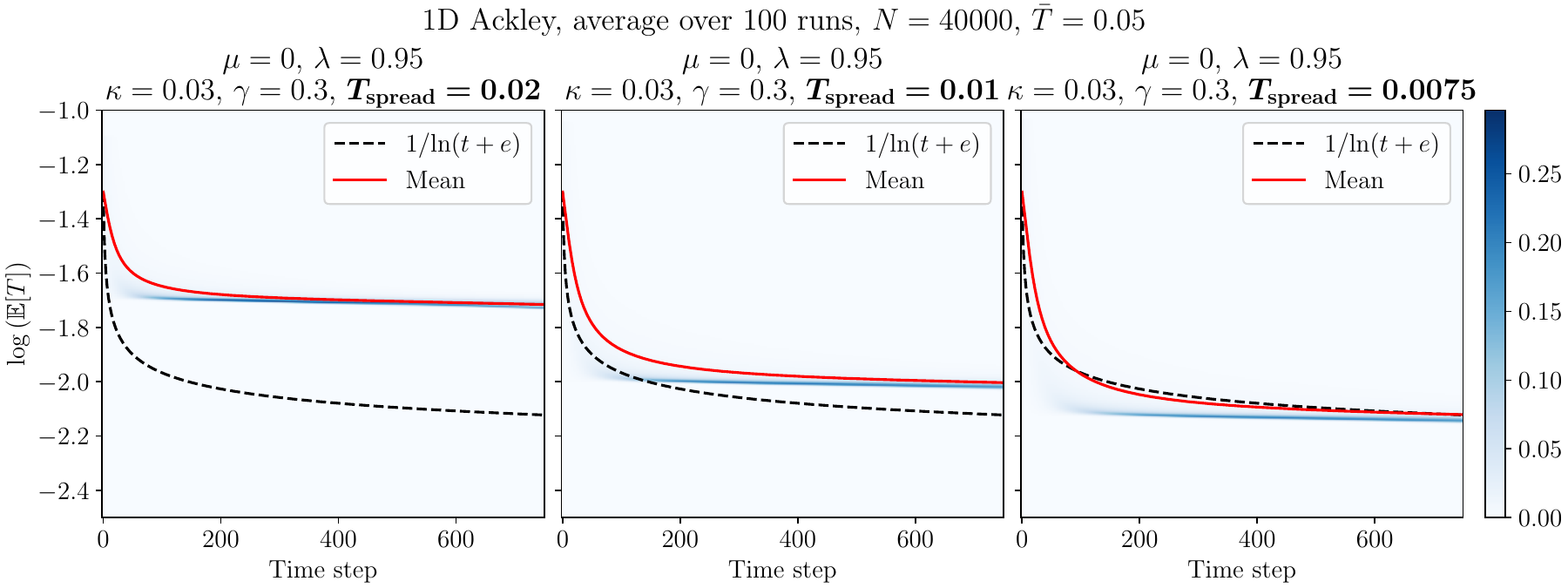}
\caption{Effect of the temperature-spread parameter $T_{\mathrm{spread}}$.}
\label{fig:tvar}
\end{figure}
As shown in Figure~\ref{fig:tvar}, decreasing the temperature-spread
parameter \(T_{\mathrm{spread}}\) leads to the following observations:
\begin{itemize}
\item Decreasing \(T_{\mathrm{spread}}\) increases the initial spread of the
temperature distribution while keeping its arithmetic mean fixed at
\(T_s\). Indeed, with the choice
\[
T_h'=2T_s-T_{\mathrm{spread}},\qquad T_l'=T_{\mathrm{spread}},
\]
one has \((T_h'+T_l')/2=T_s\). Hence, the lower mean temperature observed
during the evolution is not due to a lower initial average temperature,
but rather to the subsequent interaction dynamics. A wider initial
temperature range creates larger temperature gaps between interacting
particles, which can enhance the cooling effect generated by the
temperature exchange mechanism.

\item For \(T_{\mathrm{spread}}=0.0075\), the resulting mean-temperature
profile gives a close approximation of the logarithmic decay. Although
part of the temperature distribution lies below the reference logarithmic
schedule, a sufficient fraction of particles remains at higher
temperatures, preserving exploration while the empirical mean follows the
desired cooling trend.
\end{itemize}

\paragraph{Conclusion.}
This experiment highlights the role of the main CAST parameters in shaping 
the decay of the empirical mean temperature. Although the parameters were 
chosen by manual tuning and are not claimed to be optimal, the results show 
that CAST can produce a temperature evolution that closely follows an inverse 
logarithmic cooling schedule. This indicates that the proposed interaction 
mechanism is sufficiently flexible to recover classical cooling behaviours 
through a purely collective temperature dynamics.

\subsubsection{Geometric decay}
We now wish to approximate geometric decay, i.e. $0.999^t$. We can play the same game as before, using our acquired knowledge of the parametric effects, to find the following approximations, shown in Figure \ref{fig:geo}.

\begin{figure}[htb]
\centering
\includegraphics[width=0.7\linewidth]{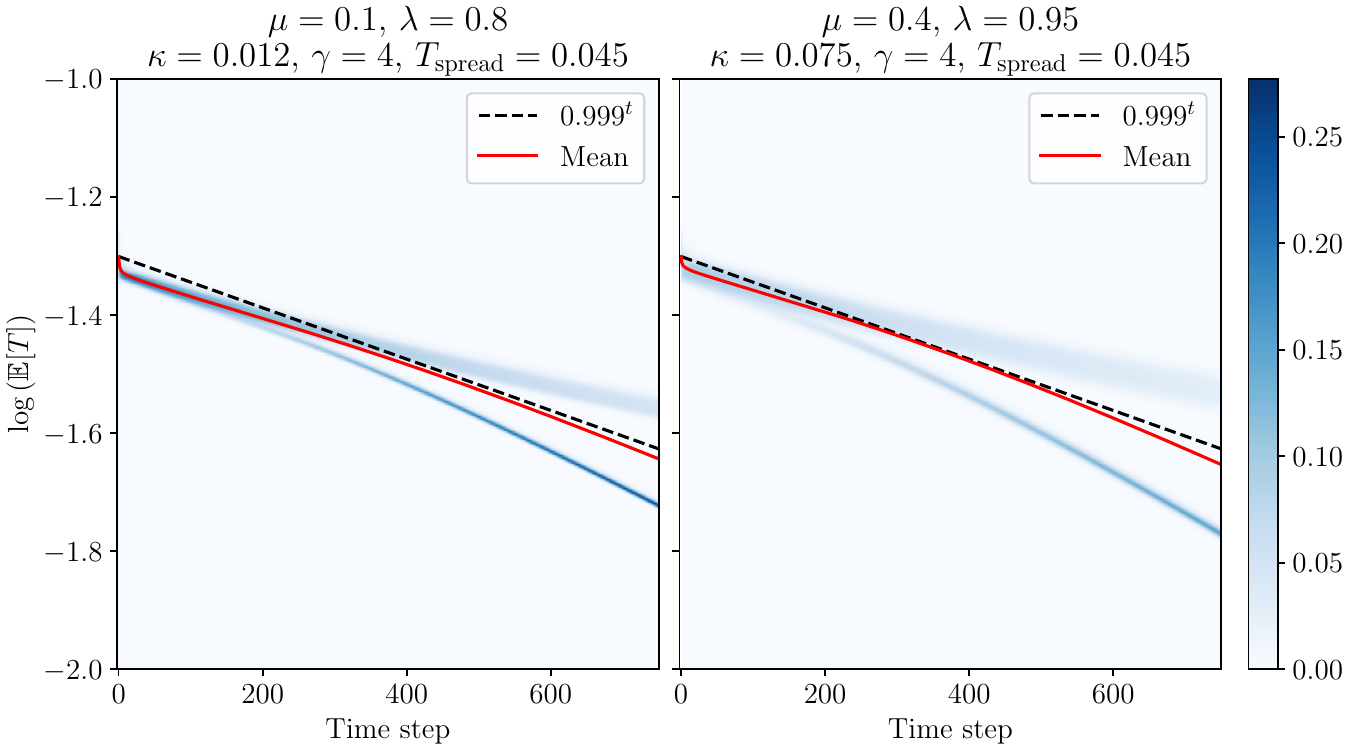}
\caption{Approximation of geometric decay.}
\label{fig:geo}
\end{figure}

Here we see two similar looking decays but with \textit{different} parameter choices. A few key observations:
\begin{itemize}
\item To achieve the rapid decay that is characteristic of geometric decay, we must increase $\gamma$ significantly compared to inverse logarithmic decay. Here, setting $\gamma=4$ means that particles will exchange temperatures four times as much compared to a ``normal" simulation.
\item On the left panel, we set a relatively small value for $\mu$ like we did with the inverse logarithmic decay. However, on the right panel, we see a similar decay with $\mu=0.4$, a significantly larger value. By increasing the value of $\mu$, we expect a slower decrease in mean temperature as shown in Lemma \ref{lemma:moment1}. To counteract this effect, we have three options:
\begin{enumerate}
\item Increase $\lambda$. As we can only increase it so much, we set it to $0.95$.
\item Increase $\kappa$. There are two reasons to increase this value. Firstly to speed up the temperature decay. Secondly, to counteract the effect of the previous increase of $\lambda$ since this reduced the support $a$. We set $\kappa$ to $0.08$.
\item Increase $\gamma$. For this experiment we did not touch this value, although one could do so.
\end{enumerate}

\end{itemize}

\subsection{Convergence}
\label{sec:conv}
Now, we want to verify whether we get similar convergence rates as SA. We increase the dimensionality to $10$ and simulate $100$ particles over $5000$ time steps. We repeat and average the simulation $100$ times to avoid statistical bias. 

On the following figures, we will display the inverse logarithmic decay of classical simulated annealing as a solid blue line on the left panel. The solid and dashed red lines are the arithmetic mean and the geometric mean of the temperatures respectively. The dot-dashed red line is the temperature of the best particle $\widehat{X}^n$. On the right upper panel, we will show the average log MSE of the best particle for both SA and CAST (see eq. (\ref{eq:best})). On the right lower panel we will show the average log MSE of the average particle position for both SA and CAST (see eq. (\ref{eq:avg})).

\paragraph{Inverse logarithmic decay} Since dimensionality affects the algorithm, we slightly tweak the values we found for the inverse logarithmic decay. We take the parameter set as follows
\begin{equation*}
(\mu,\lambda,\kappa,\gamma,T_{\mathrm{spread}})=(0,0.95,0.02,0.3,0.0075).
\end{equation*}

\begin{figure}[htb]
\centering
\includegraphics[width=0.7\linewidth]{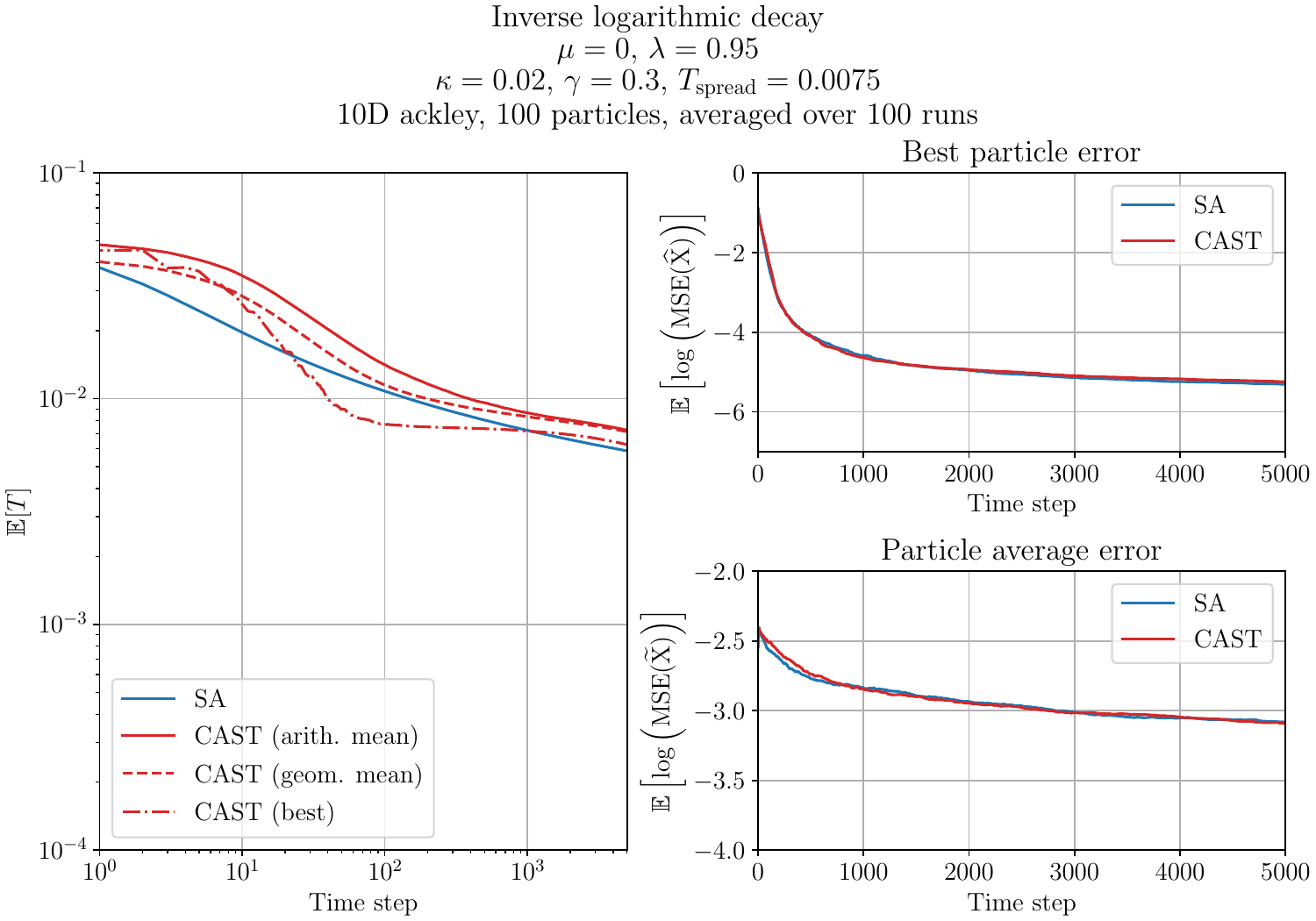}
\caption{Convergence comparison between SA with inverse logarithmic decay and CAST.}
\label{fig:log_conv}
\end{figure}

Shown in Figure \ref{fig:log_conv}, we observe that the convergence rate for both the best particle error and the particle average error is identical for SA and CAST in the case of inverse logarithmic decay. 

\paragraph{Geometric decay} Like before, we tweak the parameters somewhat due to the higher dimensionality. Starting with the parameter set of the left panel of Figure \ref{fig:geo}, we slightly increase $\kappa$ and $\gamma$. We take
\begin{equation*}
(\mu,\lambda,\kappa,\gamma,T_{\mathrm{spread}})=(0.1,0.8,0.02,4.1,0.045).
\end{equation*}

\begin{figure}[htb]
\centering
\includegraphics[width=0.7\linewidth]{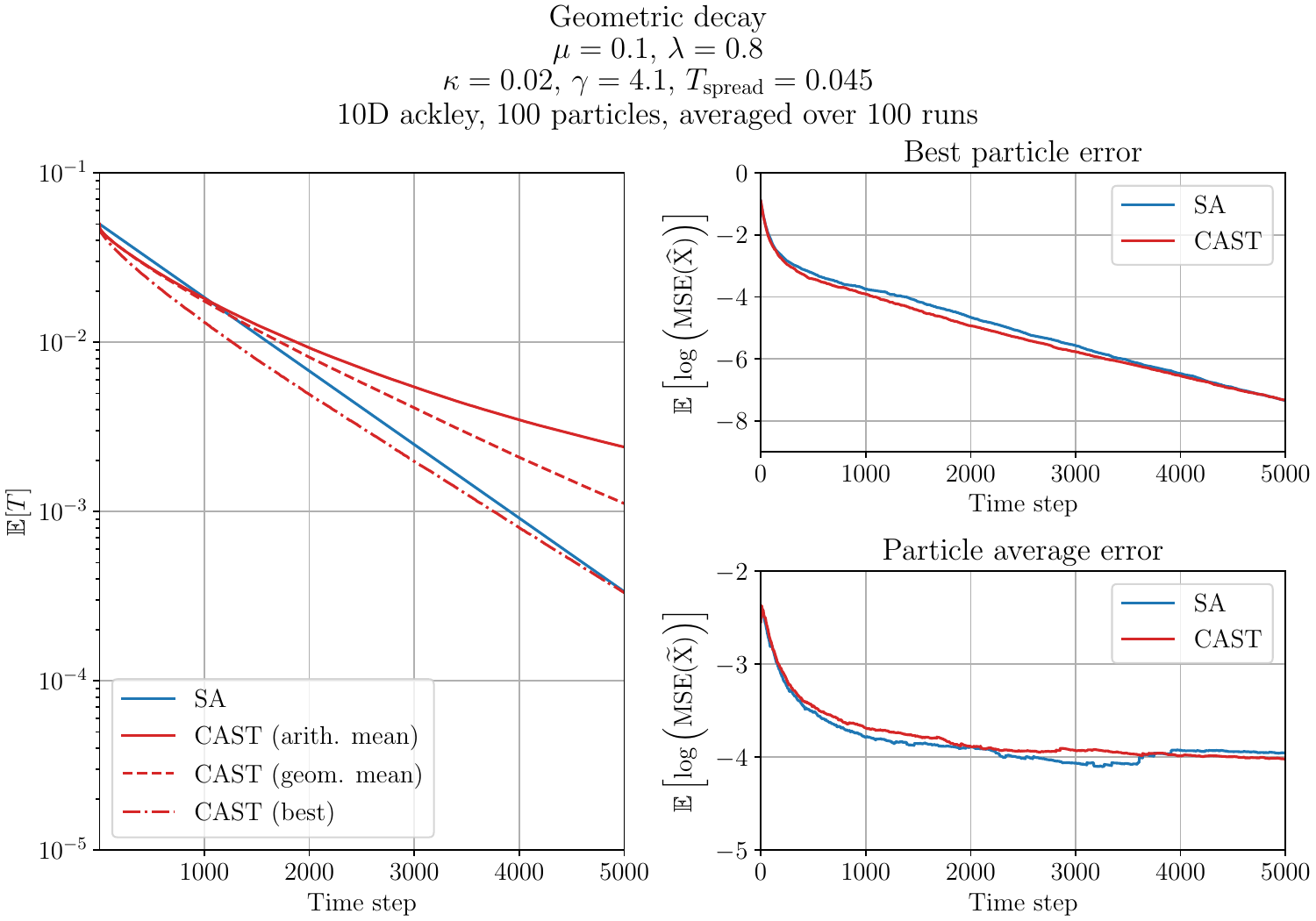}
\caption{Convergence comparison between SA with geometric decay and CAST.}
\label{fig:geo_conv1}
\end{figure}

Shown in Figure \ref{fig:geo_conv1}, we observe the following:
\begin{itemize}
\item We observe very similar empirical convergence behaviour for SA and CAST,
indicating that CAST can reproduce the effect of a geometric cooling schedule
after parameter tuning.

\item The decay of the temperature is not \textit{perfectly} geometric. We observe a slight upward trend in time. This could be mitigated by increasing $\gamma$. We did not do so because the \textit{initial} decay of CAST is slightly faster than geometric. 
\end{itemize}

The convergence plots for the other set of parameters found on the right panel of Figure \ref{fig:geo} are essentially identical and we omit them for brevity. 

\subsection{Effective parameters}

In order to use CAST efficiently, we want to find the effective parameters. Reminded by the \textit{No Free Lunch Theorem},\cite{wolpert_no_1997} we cannot expect a parameter set that will work perfectly for every test function. However, we try to find the set that ``fits best" the two test functions. To this end, we will perform four hyperparameters searches, summarized in the following tables. We group the hyperparameters in two categories: main and generalization hyperparameters. We will start by optimizing the latter and then the former since $\mu$ and $\lambda$ are arguably the most important parameters. The notation $[a,b]_{type=n}$ indicates that the $n$ parameter values are chosen between $a$ and $b$, and are either linearly or logarithmically distributed. 
\\\\
\begin{table}[htb]
\centering
\begin{tabular}[t]{l|lcccc}
\hline
\textbf{}&&$N$&$\max \ T_\mathrm{steps}$&$\mu$&$\lambda$\\
\hline
\multirow{2}{*}{Ackley}&5D&2000&$1000$&$[0,1]_{lin=21}$&$[0,1]_{lin=21}$\\
&10D&400&$1000$&$[0,1]_{lin=21}$&$[0,1]_{lin=21}$\\
\hline\hline
\multirow{2}{*}{Rastrigin}&5D&2000&$5000$&$[0,1]_{lin=21}$&$[0,1]_{lin=21}$\\
&10D&400&$20000$&$[0,1]_{lin=16}$&$[0,1]_{lin=16}$\\
\hline
\end{tabular}
\caption{Main hyperparameters.}
\end{table}
\vspace{-1.5em}
\begin{table}[htb]
\centering
\begin{tabular}[t]{l|lccc}
\hline
\textbf{}&&$T_{\mathrm{spread}}$&$\kappa$&$\gamma$\\
\hline
\multirow{2}{*}{Ackley}&5D&$[0.005, 0.01, 0.02, 0.04]$&$[0.01, 1]_{log=20}$&$[0.1, 3.16]_{log=15}$\\
&10D&$[0.005, 0.01, 0.02, 0.04]$&$[0.01, 1]_{log=20}$&$[0.1, 3.16]_{log=15}$\\
\hline\hline
\multirow{2}{*}{Rastrigin}&5D&$[0.005, 0.01, 0.02, 0.04]$&$[0.01, 1]_{log=20}$&$[0.1, 3.16]_{log=15}$\\
&10D&$[0.005, 0.01]$&$[0.01, 1]_{log=15}$&$[0.1, 3.16]_{log=10}$\\
\hline
\end{tabular}
\caption{Generalization hyperparameters.}
\end{table}
\\\\
To compare performance, we will look at two metrics:
\begin{itemize}
\item \textbf{Success rate}. If the best particle $\widehat{X}^n$ manages to reach the basin of attraction (in practice, we halve the basin radius like it is done in much of the literature), we count this as a success.
\item \textbf{Steps to basin}. The amount of steps needed for the best particle $\widehat{X}^n$ to reach the basin of attraction.
\end{itemize}

\subsubsection{Optimizing noise, interaction rate and temperature variance}

To understand the algorithm better, we want to observe the effect of each parameter separately. In order to do so, we marginalize over the other parameters. For example, if we have some metric in the form
\begin{equation*}
M(T_{\mathrm{spread}},\kappa,\gamma),
\end{equation*}
we can get an idea of the metric in function of only $T_{\mathrm{spread}}$ by computing
\begin{equation*}
M_1(T_{\mathrm{spread}})=\frac{1}{N_\kappa}\frac{1}{N_\gamma}\sum_{\kappa, \gamma}M(T_{\mathrm{spread}},\kappa,\gamma),
\end{equation*}
where $N_\kappa$ and $N_\gamma$ are the amount of samples in $\kappa$ and $\gamma$ respectively. The metric we will consider is the ``Normalised weighted average steps". To compute this, we divide the average steps by the success rate and normalize between 0 and 1. 

\begin{figure}[htb]
\centering
\includegraphics[width=0.9\linewidth]{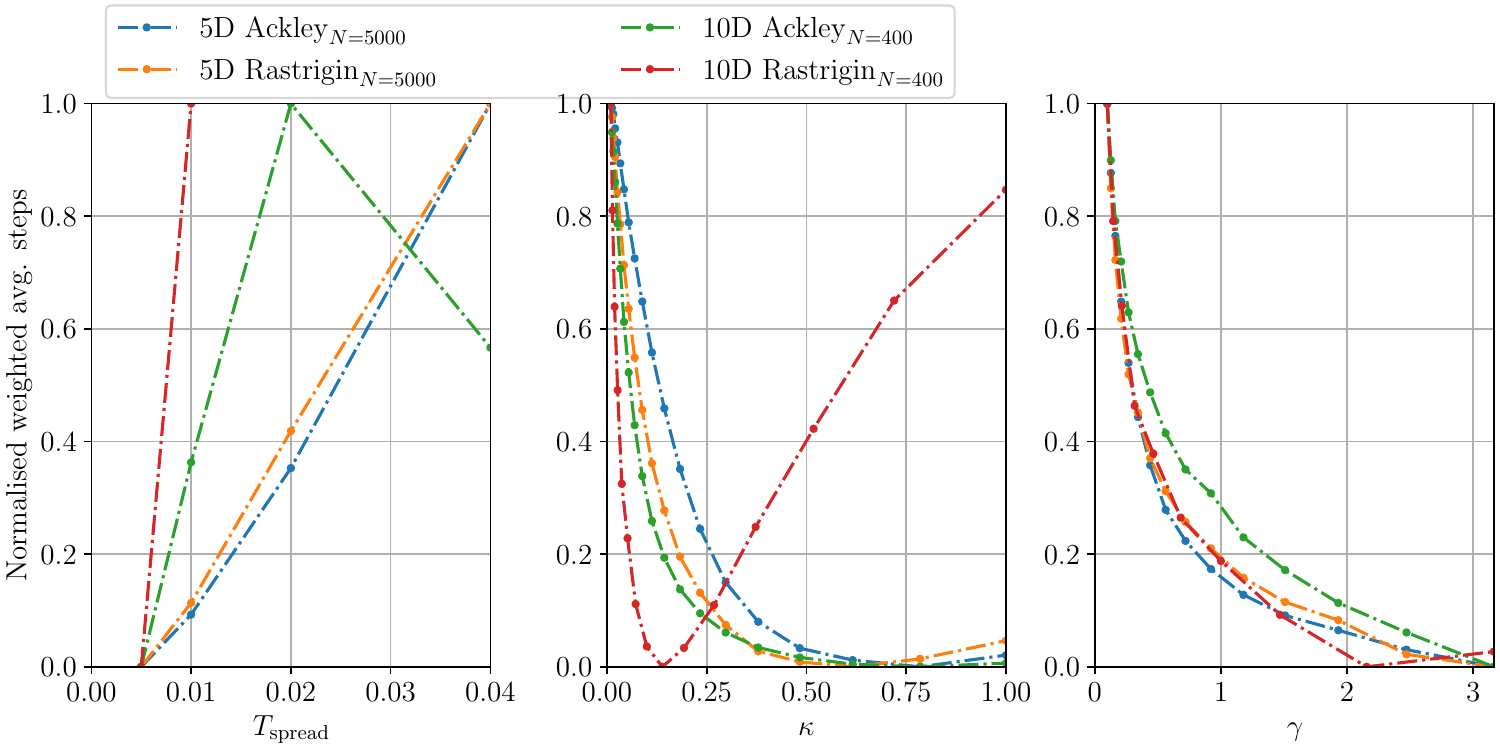}
\caption{Effects of $T_{\mathrm{spread}}$, $\kappa$ and $\gamma$ on the performance of the algorithm.}
\label{fig:norm}
\end{figure}

This gives Figure \ref{fig:norm}, where we can make the following observations:

\begin{itemize}
\item Even for different functions, dimensions and particle counts, we observe that the performance of the algorithm behaves similarly. This is evident from Figure \ref{fig:norm}. The only exception is the behaviour of 10D Rastrigin for high values of $\kappa$.

\item The effect of $\kappa$ is qualitatively similar across the test cases, with a clear turning point in each curve. For 10D Rastrigin, the turning point comes a lot earlier than in the other test cases. This is likely due to the definition of Rastrigin itself. Looking at eq. (\ref{eq:rastrigin}), the function evaluations are \textit{dependent} on $d$. This has the consequence of ``stretching" the function out in higher dimensions. This causes particles to become more easily stuck in local minima since the ``jump out" becomes more difficult. Higher values for $\kappa$ implies faster cooling, which implies that particles are expected to find the global minimizer fast. The lower effective value of \(\kappa\) indicates that slower cooling is
preferable in this more difficult high-dimensional regime.

\item Since we average over the remaining parameters to visualize each parameter
separately, these curves should be interpreted as qualitative indicators.
In practice we should take smaller values than the ones we see in Figure \ref{fig:norm}.
\end{itemize}

We are now able to choose the first three parameters:
\begin{align*}
(\kappa,\gamma,T_{\mathrm{spread}})_{5\mathrm{D}}=(0.35,2,0.005) \quad
(\kappa,\gamma,T_{\mathrm{spread}})_{10\mathrm{D}}=(0.15,1.5,0.005).
\end{align*}

\subsubsection{Optimizing the temperature interaction parameters}

After fixing the parameters $\kappa$, $\gamma$ and $T_{\text{spread}}$, we select the interaction parameters $\mu$ and $\lambda$ by a grid search based on the success rate and on the average number of steps needed to reach the basin of attraction. We display the results for the four test cases in Figures \ref{fig:mulam5d} and \ref{fig:mulam10d}.
\begin{figure}[!htb]
\centering
\begin{subfigure}{1\textwidth}
\centering
\includegraphics[width=0.7\linewidth]{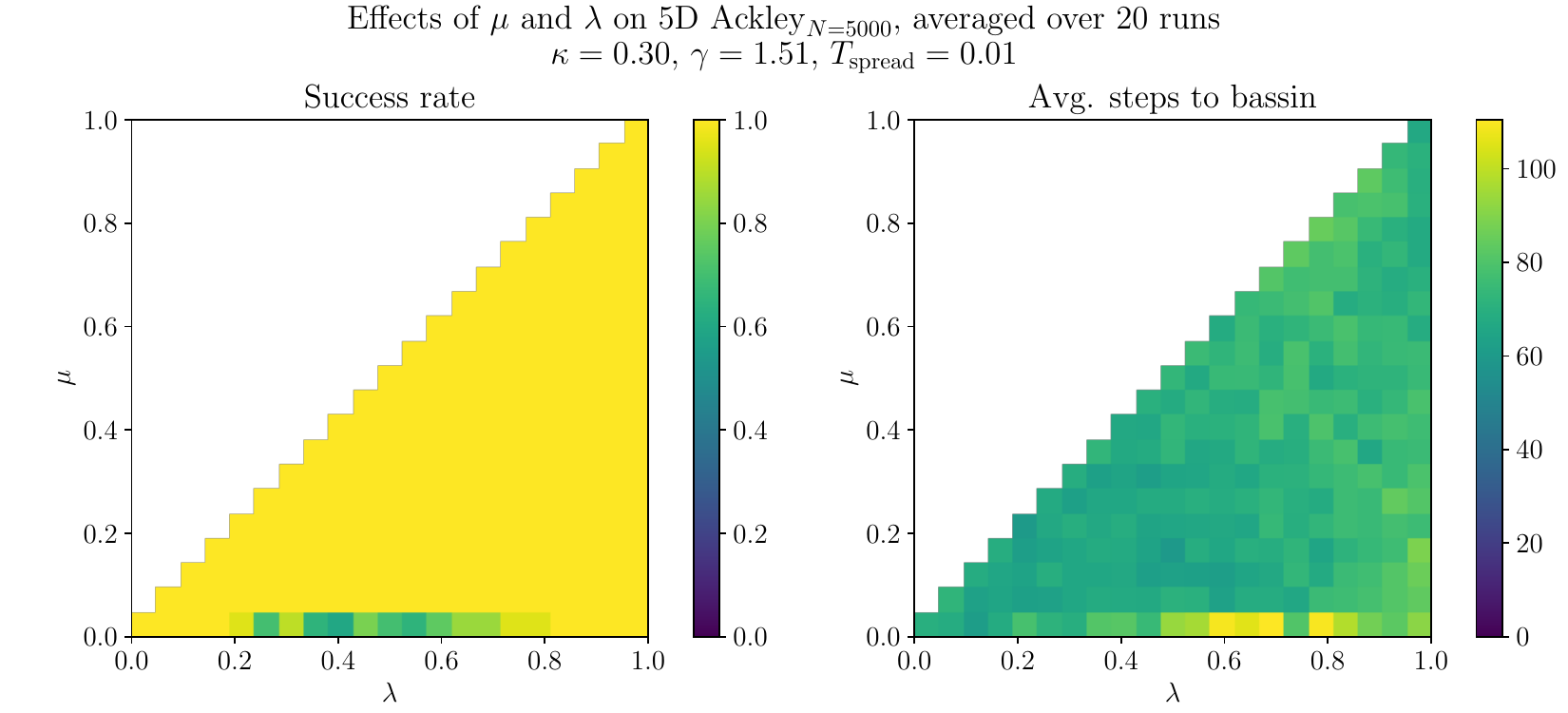}
\caption{5D Ackley}
\label{fig:ack5d}
 \end{subfigure}
 \begin{subfigure}{1\textwidth}
 \centering
\includegraphics[width=0.7\linewidth]{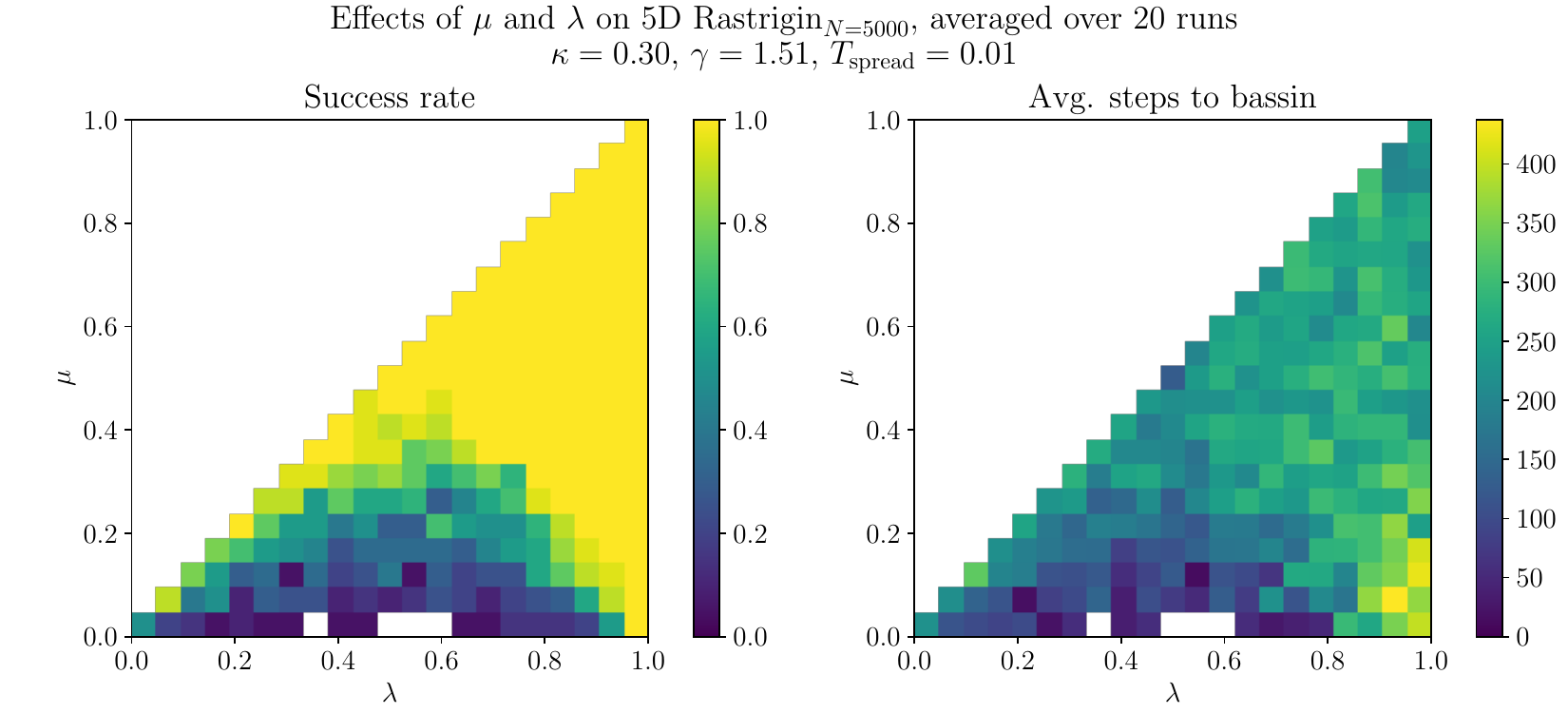}
\caption{5D Rastrigin}
\label{fig:rastr5d}
 \end{subfigure}
\caption{Effects of $\mu$ and $\lambda$.}
\label{fig:mulam5d}
\end{figure}
\begin{figure}[!htb]
\centering
\begin{subfigure}{1\textwidth}
 \centering
\includegraphics[width=0.7\linewidth]{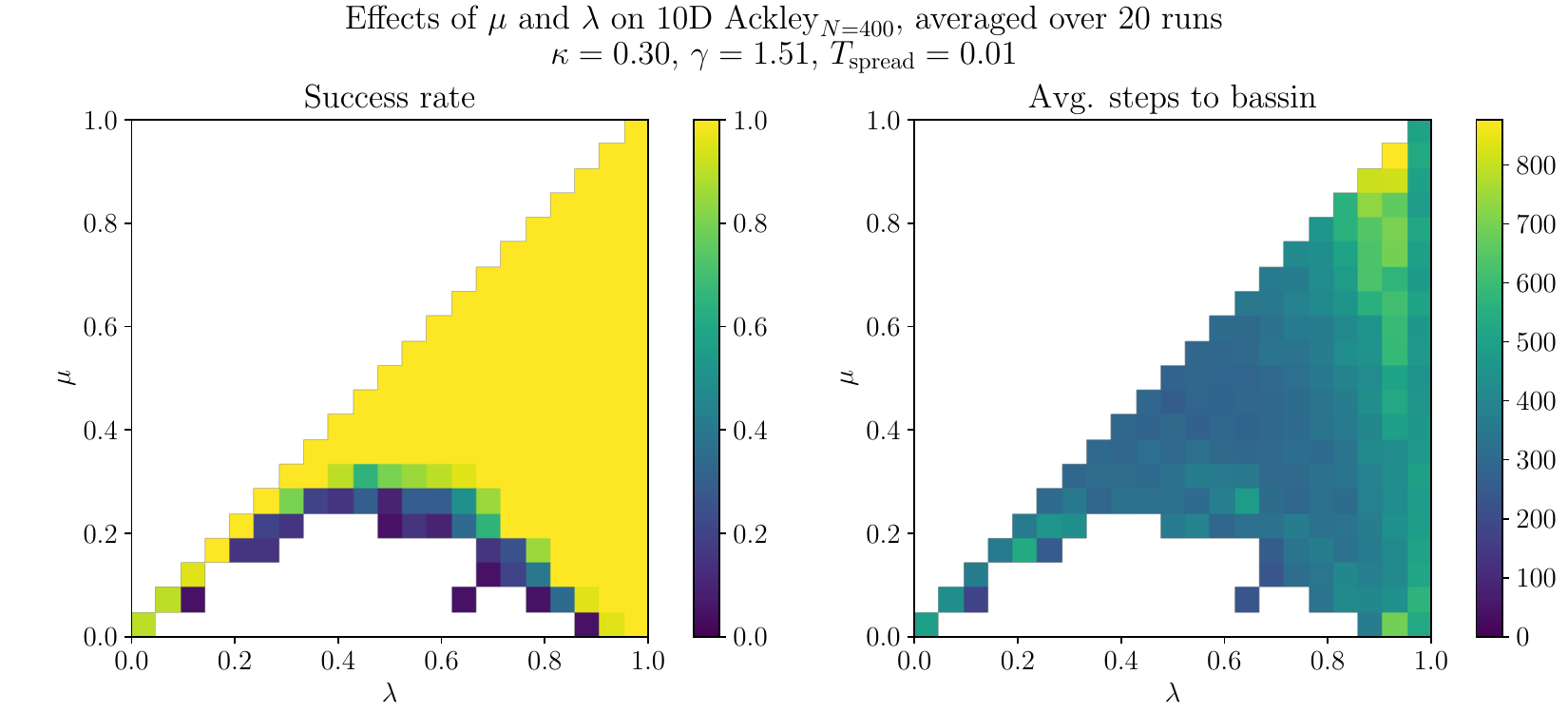}
\caption{10D Ackley}
\label{fig:ack10d}
 \end{subfigure}
 \begin{subfigure}{1\textwidth}
  \centering
\includegraphics[width=0.7\linewidth]{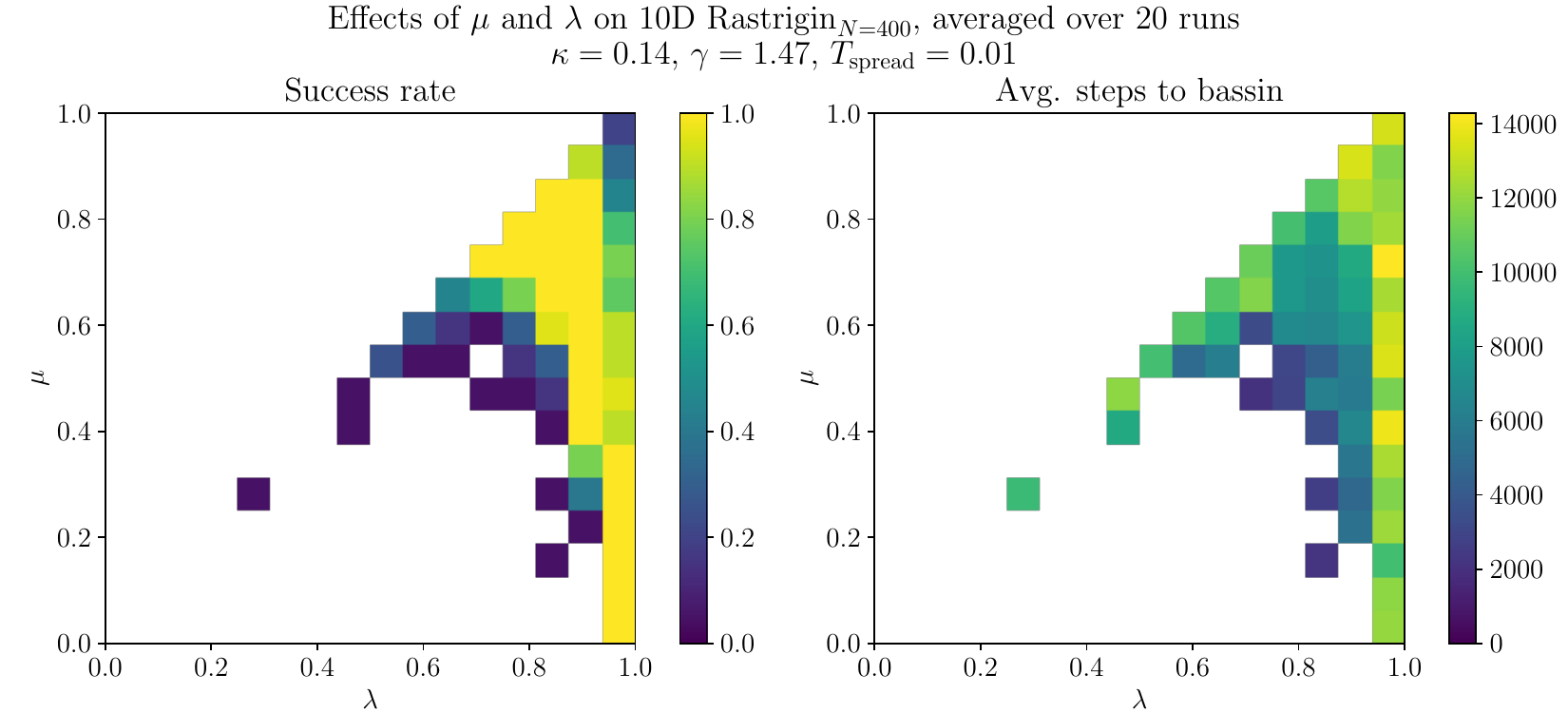}
\caption{10D Rastrigin}
\label{fig:rastr10d}
 \end{subfigure}
\caption{Effects of $\mu$ and $\lambda$.}
\label{fig:mulam10d}
\end{figure}
The spots without color correspond to a success rate of zero, i.e., a total failure \textit{or} if $\mu>\lambda$. Ideally, we want to pick a parameter set that has a success rate of 1 and a low amount of average steps. We settle on the following parameters:
\begin{align*}
5\mathrm{D}\rightarrow&(\mu,\lambda,\kappa,\gamma,T_{\mathrm{spread}})=(0.5,0.7,0.35,2,0.005)\\
10\mathrm{D}\rightarrow&(\mu,\lambda,\kappa,\gamma,T_{\mathrm{spread}})=(0.65,0.85,0.15,1.5,0.005).
\end{align*}

It shouldn't surprise that for more difficult functions in higher dimensions and with fewer particles, fewer good parameter sets exist. If we compare Figures \ref{fig:ack5d} and \ref{fig:rastr10d}, it makes sense that we have a lot more good choices for an ``easier function" in lower dimension with more particles than vice-versa.

\subsubsection{Comparison with simulated annealing}

Having selected effective hyperparameter sets, we can compare CAST to classical SA. We will test two scenarios:
\begin{itemize}
\item $5\mathrm{D}$ optimization with $2000$ particles. This is an ``intermediate" dimensional problem with many particles.
\item $10\mathrm{D}$ optimization with $400$ particles. This is a ``higher" dimensional problem with relatively few particles (especially for that dimensionality).
\end{itemize}
For each scenario, we run the optimization algorithms over $100$ independent runs to reduce
Monte Carlo fluctuations.

\paragraph{5D comparison} We present the expected log MSE of SA and CAST's best particle in each time step in  Figure \ref{fig:5dconvcomp}. On the left panel for Ackley and on the right panel for Rastrigin.

\begin{figure}[!htb]
\centering
\begin{subfigure}[b]{1\textwidth}
\centering
\includegraphics[width=0.75\linewidth]{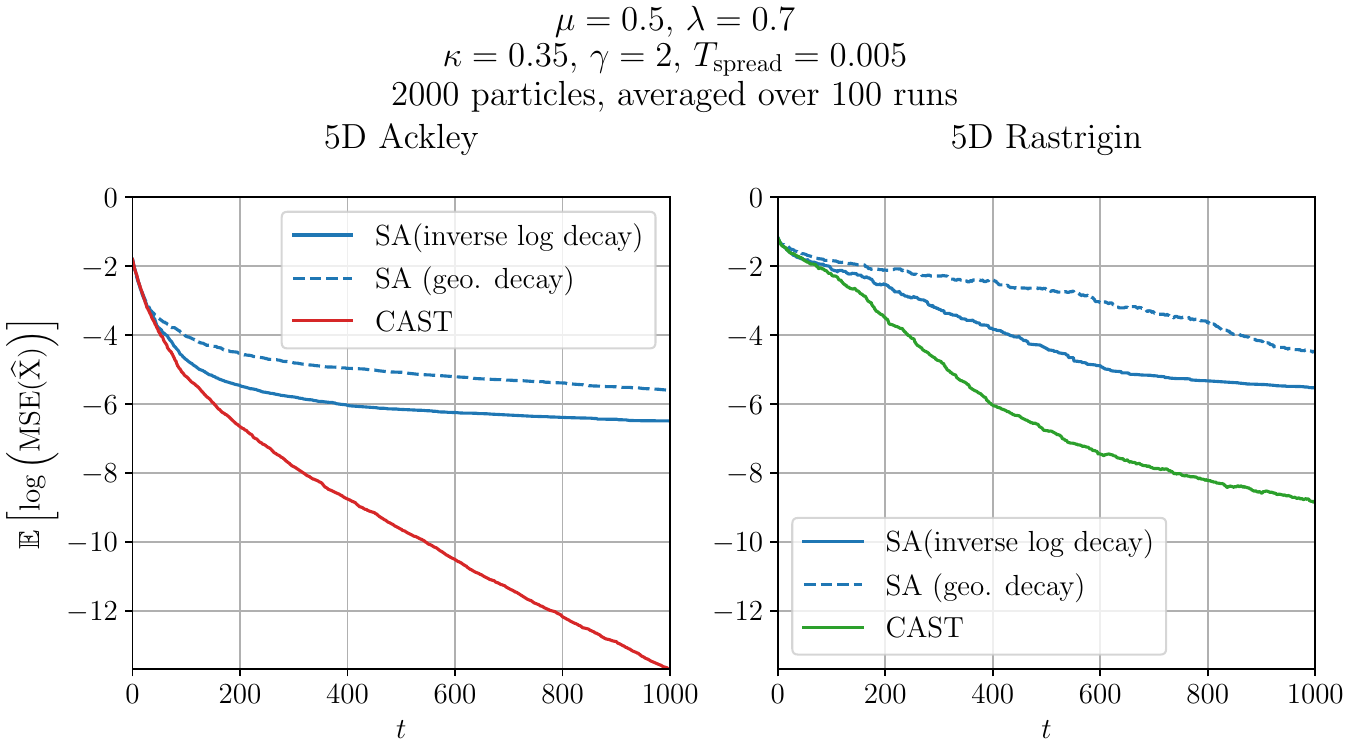}
\caption{Convergence of the expected log MSE.}
\label{fig:5dconvcomp}
\end{subfigure}
\begin{subfigure}[b]{1\textwidth}
\centering
\includegraphics[width=0.75\linewidth]{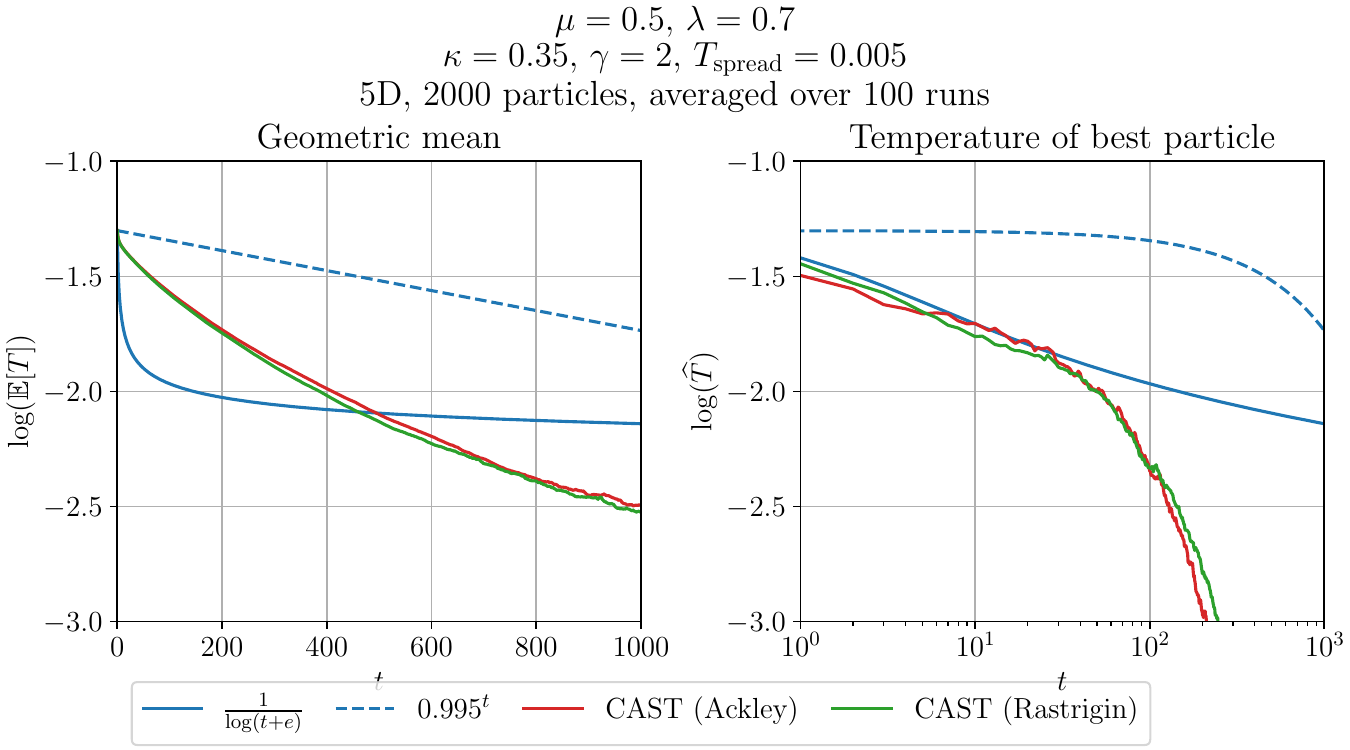}
\caption{Temperature decay.}
\label{fig:5dtemp}
\end{subfigure}
\caption{5D comparison}
\end{figure}

Our algorithm clearly converges much faster than classical SA for  both test functions, but especially for Rastrigin. It is interesting to see that the error curves for SA and CAST lie quite close during the initial optimization phase after which they diverge. On Figure \ref{fig:5dtemp}, we show the temperature curves. As reference, we will also include the shapes of inverse logarithmic and geometric decays with the solid and dashed blue lines. We observe the following:
\begin{itemize}
\item On the left panel, we see the mean temperature curves for both Ackley and Rastrigin follow each other closely. This seems to suggest that the decay is relatively insensitive to the underlying function itself.
\item On the right panel (switching to a loglog plot), we observe that the best particle's temperature initially follows an inverse logarithmic decay before accelerating and following something that is more geometric in nature. This seems to suggest that the algorithm ``adapts'' during the simulation. It is likely that this acceleration occurs when the best particle enters the basin of attraction.
\end{itemize}

\paragraph{10D comparison} \textbf{}
\begin{figure}[!htb]
\centering
\begin{subfigure}[b]{1\textwidth}
\centering
\includegraphics[width=0.75\linewidth]{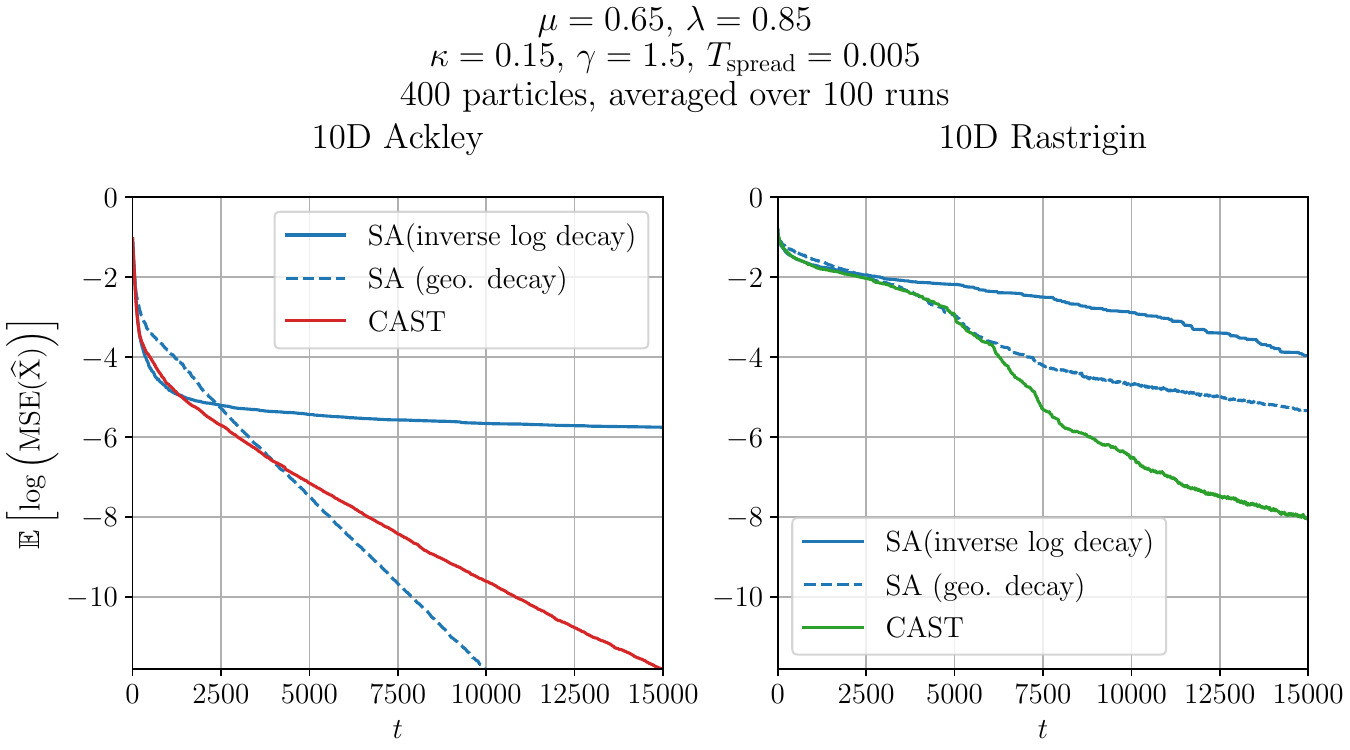}
\caption{Convergence of the expected log MSE.}
\label{fig:10dconvcomp}
\end{subfigure}
\begin{subfigure}[b]{1\textwidth}
\centering
\includegraphics[width=0.75\linewidth]{b451eaee5ad060629fffed390f2a7bd1_conv_comp.pdf}
\caption{Temperature decay.}
\label{fig:10dtemp}
\end{subfigure}
\caption{10D comparison}
\end{figure}
Interestingly, while we observe similar characteristics for this scenario as in the previous one, we notice that SA with geometric decay converges faster for 10D Ackley than CAST. This suggests that Ackley allows faster cooling. However, it does not outperform CAST for 10D Rastrigin. On Figure \ref{fig:10dtemp}, we look at the temperature decay where we observe the following:
\begin{itemize}
\item The means of the temperatures for Ackley and Rastrigin follow each other like before. However, we see that they diverge near the end of the optimization run. The interpretation here would be as follows: since Rastrigin is a more difficult optimization problem with many local minima (lying close to the global one), there is a need for more ``warm'' particles exploring the search space.
\item On the right panel, looking at the temperature of the best particle, we observe even stronger evidence for the claim we made earlier in the 5D test case. We suggested that during the initial phase of the run, the best temperature seemed to follow an inverse logarithmic decay while after some time, the temperature starts decaying geometrically. We observe the same phenomenon here.
\end{itemize}

\section{Concluding remarks}
In this work, we have introduced and studied a novel framework for simulated annealing based on kinetic theory and collective temperature dynamics. Inspired by particle interactions in Boltzmann-type models and the principles of parallel tempering, our approach constructs a stochastic mechanism through which particles interact by exchanging temperatures according to a suitably designed binary interaction. This formulation leads to a linear kinetic model for simulated annealing of the type introduced in^^>\cite{pareschi_optimization_2024} coupled with a nonlinear Boltzmann--Povzner equation for the temperature exchanges.
The method enables a dynamic adaptation of the temperature distribution that is emergent rather than externally prescribed, and we have shown that this interaction-driven mechanism yields a monotone decay of
the expected mean temperature under suitable assumptions on the interaction
parameters. The numerical experiments support the use of CAST as an adaptive
collective cooling strategy for global optimization.

This perspective complements and extends other recent kinetic approaches to simulated annealing, such as those based on entropy-controlled cooling^^>\cite{herty_kinetic_2026}, and opens several promising directions for future research. In particular, the introduction of a vector-valued temperature could allow for anisotropic exploration of the state space, where different directions or features of the problem are cooled at different rates. This would be especially useful in high-dimensional or stiff problems, where energy landscapes vary strongly across directions.

Furthermore, the generality of the kinetic formulation suggests potential applications beyond classical optimization. For instance, collective thermostatting mechanisms could be fruitfully integrated into stochastic gradient Langevin dynamics (SGLD) or variational inference frameworks in machine learning, where adaptive control of exploration and convergence plays a central role. Finally, extending the model to allow for spatially structured particle systems or interacting agent networks may yield insights into distributed optimization and decentralized learning algorithms.

\subsection*{Acknowledgments}
The authors acknowledge financial support from the European Union’s Horizon Europe research and innovation
program under the Marie Sklodowska-Curie Doctoral Network DataHyking (Grant No. 101072546).
FB acknowledges the support of the Mathematics Department of the University of Ferrara
through access to their computational resources. The research of LP has been supported by the Royal
Society under the Wolfson Fellowship “Uncertainty quantification, data-driven simulations and learning
of multiscale complex systems governed by PDEs”. LP also acknowledges the partial support by Fondo
Italiano per la Scienza (FIS2023-01334) advanced grant ADAMUS. This work has been written within
the activities of GNCS group of INdAM (Italian National Institute of High Mathematics).

\bibliographystyle{plain}
\bibliography{../references}

\end{document}